\documentclass[numbers=enddot,12pt,final,onecolumn,notitlepage]{scrartcl}%
\usepackage[headsepline,footsepline,manualmark]{scrlayer-scrpage}
\usepackage[all,cmtip]{xy}
\usepackage{amsfonts}
\usepackage{amssymb}
\usepackage{amsmath}
\usepackage{amsthm}
\usepackage{framed}
\usepackage{comment}
\usepackage{color}
\usepackage[breaklinks=true]{hyperref}
\usepackage[sc]{mathpazo}
\usepackage[T1]{fontenc}
\usepackage{tikz}
\usepackage{needspace}
\usepackage{tabls}
\providecommand{\U}[1]{\protect\rule{.1in}{.1in}}
\usetikzlibrary{arrows}
\newcounter{exer}

\theoremstyle{definition}
\newtheorem{theo}{Theorem}[section]
\newenvironment{theorem}[1][]
{\begin{theo}[#1]\begin{leftbar}}
{\end{leftbar}\end{theo}}
\newtheorem{lem}[theo]{Lemma}
\newenvironment{lemma}[1][]
{\begin{lem}[#1]\begin{leftbar}}
{\end{leftbar}\end{lem}}
\newtheorem{prop}[theo]{Proposition}
\newenvironment{proposition}[1][]
{\begin{prop}[#1]\begin{leftbar}}
{\end{leftbar}\end{prop}}
\newtheorem{defi}[theo]{Definition}
\newenvironment{definition}[1][]
{\begin{defi}[#1]\begin{leftbar}}
{\end{leftbar}\end{defi}}
\newtheorem{remk}[theo]{Remark}

\newtheorem{coro}[theo]{Corollary}
\newenvironment{corollary}[1][]
{\begin{coro}[#1]\begin{leftbar}}
{\end{leftbar}\end{coro}}
\newtheorem{conv}[theo]{Convention}

\newtheorem{quest}[theo]{Question}
\newenvironment{question}[1][]
{\begin{quest}[#1]\begin{leftbar}}
{\end{leftbar}\end{quest}}
\newtheorem{warn}[theo]{Warning}

\newtheorem{conj}[theo]{Conjecture}

\newtheorem{exam}[theo]{Example}
\newenvironment{example}[1][]
{\begin{exam}[#1]\begin{leftbar}}
{\end{leftbar}\end{exam}}
\newtheorem{exmp}[exer]{Exercise}

\let\sumnonlimits\sum
\let\prodnonlimits\prod
\let\cupnonlimits\bigcup
\let\capnonlimits\bigcap
\renewcommand{\sum}{\sumnonlimits\limits}
\renewcommand{\prod}{\prodnonlimits\limits}
\renewcommand{\bigcup}{\cupnonlimits\limits}
\renewcommand{\bigcap}{\capnonlimits\limits}
\newenvironment{verlong}{}{}
\newenvironment{vershort}{}{}

\excludecomment{verlong}
\includecomment{vershort}
\excludecomment{noncompile}

\ihead{Commutators, matrices and an identity of Copeland}
\ohead{page \thepage}
\begin{document}

\title{Commutators, matrices and an identity of Copeland}
\author{Darij Grinberg}
\date{
\today
}
\maketitle

\begin{abstract}
\textbf{Abstract.} Given two elements $a$ and $b$ of a noncommutative ring, we
express $\left(  ba\right)  ^{n}$ as a \textquotedblleft row vector times
matrix times column vector\textquotedblright\ product, where the matrix is the
$n$-th power of a matrix with entries $\dbinom{i}{j}\operatorname*{ad}%
\nolimits_{a}^{i-j}\left(  b\right)  $. This generalizes a formula by Tom
Copeland used in the study of Pascal-style matrices.

\end{abstract}
\tableofcontents

\section{Introduction}

In \cite{MO337766}, Tom Copeland stated a formula for the $n$-th power of a
differential operator. Our goal in this note is to prove a more general
version of this formula, in which differential operators are replaced by
arbitrary elements of a noncommutative ring.

In a nutshell, this general result (Theorem \ref{thm.gen}) can be stated as
follows: If $n\in\mathbb{N}$ and $m\in\mathbb{N}\cup\left\{  \infty\right\}  $
satisfy $n<m$, and if $a$ and $b$ are two elements of a (noncommutative) ring
$\mathbb{L}$, then%
\[
\left(  ba\right)  ^{n}=e_{0}^{T}\left(  U_{b}S\right)  ^{n}H_{1},
\]
where the column vectors $e_{0}$ and $H_{1}$ of size $m$ are defined by%
\[
e_{0}=\left(
\begin{array}
[c]{c}%
1\\
0\\
0\\
\vdots\\
0
\end{array}
\right)  \ \ \ \ \ \ \ \ \ \ \text{and}\ \ \ \ \ \ \ \ \ \ H_{1}=\left(
\begin{array}
[c]{c}%
a^{0}\\
a^{1}\\
a^{2}\\
\vdots\\
a^{m-1}%
\end{array}
\right)  ,
\]
and where the $m\times m$-matrices $S$ and $U_{b}$ are defined by%
\begin{align*}
S  &  =\left(  \left[  j=i+1\right]  \right)  _{0\leq i<m,\ 0\leq j<m}=\left(
\begin{array}
[c]{ccccc}%
0 & 1 & 0 & \cdots & 0\\
0 & 0 & 1 & \cdots & 0\\
0 & 0 & 0 & \cdots & 0\\
\vdots & \vdots & \vdots & \ddots & \vdots\\
0 & 0 & 0 & \cdots & 0
\end{array}
\right)  \ \ \ \ \ \ \ \ \ \ \text{and}\\
U_{b}  &  =\left(
\begin{cases}
\dbinom{i}{j}\operatorname*{ad}\nolimits_{a}^{i-j}\left(  b\right)  , &
\text{if }i\geq j;\\
0, & \text{if }i<j
\end{cases}
\right)  _{0\leq i<m,\ 0\leq j<m}\\
&  =\left(
\begin{array}
[c]{ccccc}%
b & 0 & 0 & \cdots & 0\\
\operatorname*{ad}\nolimits_{a}\left(  b\right)  & b & 0 & \cdots & 0\\
\operatorname*{ad}\nolimits_{a}^{2}\left(  b\right)  & 2\operatorname*{ad}%
\nolimits_{a}\left(  b\right)  & b & \cdots & 0\\
\vdots & \vdots & \vdots & \ddots & \vdots\\
\operatorname*{ad}\nolimits_{a}^{m-1}\left(  b\right)  & \left(  m-1\right)
\operatorname*{ad}\nolimits_{a}^{m-2}\left(  b\right)  & \dbinom{m-1}%
{2}\operatorname*{ad}\nolimits_{a}^{m-3}\left(  b\right)  & \cdots & b
\end{array}
\right)
\end{align*}
(using the standard Lie-algebraic notation $\operatorname*{ad}\nolimits_{a}$
for the operator $\mathbb{L}\rightarrow\mathbb{L},\ c\mapsto ac-ca$). (We
shall introduce all these notations in more detail below.)

\subsection*{Acknowledgments}

DG thanks the Mathematisches Forschungsinstitut Oberwolfach for its
hospitality during part of the writing process.

\section{\label{sect.form}The general formula}

\subsection{\label{subsect.form.standing-not}Standing notations}

Let us start by introducing notations that will remain in place for the rest
of this note:

\begin{itemize}
\item Let $\mathbb{N}$ denote the set $\left\{  0,1,2,\ldots\right\}  $.

\item \textquotedblleft Ring\textquotedblright\ will always mean
\textquotedblleft associative ring with unity\textquotedblright. Commutativity
is not required.

\item Fix a ring $\mathbb{L}$.

\item For any two elements $a$ and $b$ of $\mathbb{L}$, we define an element
$\left[  a,b\right]  $ of $\mathbb{L}$ by%
\[
\left[  a,b\right]  =ab-ba.
\]
This element $\left[  a,b\right]  $ is called the \textit{commutator} of $a$
and $b$.

\item For any $a\in\mathbb{L}$, we define a map $\operatorname*{ad}%
\nolimits_{a}:\mathbb{L}\rightarrow\mathbb{L}$ by
\[
\left(  \operatorname*{ad}\nolimits_{a}\left(  b\right)  =\left[  a,b\right]
\ \ \ \ \ \ \ \ \ \ \text{for all }b\in\mathbb{L}\right)  .
\]
Clearly, this map $\operatorname*{ad}\nolimits_{a}$ is $\mathbb{Z}$-linear.
\end{itemize}

\subsection{\label{subsect.form.matrices}Conventions about matrices}

In the following, we will use matrices. We shall use a slightly nonstandard
convention for labeling the rows and the columns of our matrices: Namely, the
rows and the columns of our matrices will always be indexed starting with $0$.
That is, a $k\times\ell$-matrix (for $k\in\mathbb{N}$ and $\ell\in\mathbb{N}$)
will always have its rows numbered $0,1,\ldots,k-1$ and its columns numbered
$0,1,\ldots,\ell-1$. In other words, a $k\times\ell$-matrix is a family
$\left(  a_{i,j}\right)  _{0\leq i<k,\ 0\leq j<\ell}$ indexed by pairs
$\left(  i,j\right)  $ of integers satisfying $0\leq i<k$ and $0\leq j<\ell$.
We let $\mathbb{L}^{k\times\ell}$ denote the set of all $k\times\ell$-matrices
with entries in $\mathbb{L}$.

If $A$ is any $k\times\ell$-matrix (where $k$ and $\ell$ belong to
$\mathbb{N}$), and if $i$ and $j$ are any two integers satisfying $0\leq i<k$
and $0\leq j<\ell$, then we let $A_{i,j}$ denote the $\left(  i,j\right)  $-th
entry of $A$. Thus, any $k\times\ell$-matrix $A$ satisfies%
\[
A=\left(
\begin{array}
[c]{cccc}%
A_{0,0} & A_{0,1} & \cdots & A_{0,\ell-1}\\
A_{1,0} & A_{1,1} & \cdots & A_{1,\ell-1}\\
\vdots & \vdots & \ddots & \vdots\\
A_{k-1,0} & A_{k-1,1} & \cdots & A_{k-1,\ell-1}%
\end{array}
\right)  .
\]

If $k\in\mathbb{N}$, then a \textit{column vector of size }$k$ means a
$k\times1$-matrix. Thus, a column vector of size $k$ has the form $\left(
\begin{array}
[c]{c}%
a_{0}\\
a_{1}\\
\vdots\\
a_{k-1}%
\end{array}
\right)  _{0\leq i<k,\ 0\leq j<1}$. Row vectors are defined similarly.

As usual, we shall equate $1\times1$-matrices $A\in\mathbb{L}^{1\times1}$ with
elements of $\mathbb{L}$ (namely, by equating each $1\times1$-matrix
$A\in\mathbb{L}^{1\times1}$ with its unique entry $A_{0,0}$). Thus, if $v$ and
$w$ are any two column vectors of size $k$, then $w^{T}v\in\mathbb{L}$.

\subsection{\label{subsect.form.inf-mat}Conventions about infinite matrices}

Furthermore, we shall allow our matrices to be infinite (i.e., have infinitely
many rows or columns or both). This will be an optional feature of our
results; we will state our claims in a way that allows the matrices to be
infinite, but if the reader is only interested in finite matrices, they can
ignore this possibility and skip Subsection \ref{subsect.form.inf-mat} entirely.

First of all, let us say a few words about how we will use $\infty$ in this
note. As usual, \textquotedblleft$\infty$\textquotedblright\ is just a symbol
which we subject to the following rules: We have $n<\infty$ and $\infty
+n=\infty-n=\infty$ for each $n\in\mathbb{N}$. Moreover, we shall use the
somewhat strange convention that $\left\{  0,1,\ldots,\infty\right\}  $
denotes the set $\mathbb{N}$ (so it does not contain $\infty$). This has the
consequence that $\left\{  0,1,\ldots,\infty-n\right\}  =\mathbb{N}$ for each
$n\in\mathbb{N}$ (since $\infty-n=\infty$).

We will use the following kinds of infinite matrices:

\begin{itemize}
\item A $k\times\infty$\textit{-matrix} (where $k\in\mathbb{N}$) has $k$ rows
(indexed by $0,1,\ldots,k-1$) and infinitely many columns (indexed by
$0,1,2,\ldots$). Such a matrix will usually be written as%
\[
\left(
\begin{array}
[c]{cccc}%
a_{0,0} & a_{0,1} & a_{0,2} & \cdots\\
a_{1,0} & a_{1,1} & a_{1,2} & \cdots\\
\vdots & \vdots & \vdots & \vdots\\
a_{k-1,0} & a_{k-1,1} & a_{k-1,2} & \cdots
\end{array}
\right)  =\left(  a_{i,j}\right)  _{0\leq i<k,\ 0\leq j<\infty}.
\]

\item A $\infty\times\ell$\textit{-matrix} (where $\ell\in\mathbb{N}$) has
infinitely many rows (indexed by $0,1,2,\ldots$) and $\ell$ columns (indexed
by $0,1,\ldots,\ell-1$). Such a matrix will usually be written as%
\[
\left(
\begin{array}
[c]{cccc}%
a_{0,0} & a_{0,1} & \cdots & a_{0,\ell-1}\\
a_{1,0} & a_{1,1} & \cdots & a_{1,\ell-1}\\
a_{2,0} & a_{2,1} & \cdots & a_{2,\ell-1}\\
\vdots & \vdots & \vdots & \vdots
\end{array}
\right)  =\left(  a_{i,j}\right)  _{0\leq i<\infty,\ 0\leq j<\ell}.
\]

\item A $\infty\times\infty$\textit{-matrix} has infinitely many rows (indexed
by $0,1,2,\ldots$) and infinitely many columns (indexed by $0,1,2,\ldots$).
Such a matrix will usually be written as%
\[
\left(
\begin{array}
[c]{cccc}%
a_{0,0} & a_{0,1} & a_{0,2} & \cdots\\
a_{1,0} & a_{1,1} & a_{1,2} & \cdots\\
a_{2,0} & a_{2,1} & a_{2,2} & \cdots\\
\vdots & \vdots & \vdots & \ddots
\end{array}
\right)  =\left(  a_{i,j}\right)  _{0\leq i<\infty,\ 0\leq j<\infty}.
\]

\end{itemize}

Matrices of these three kinds (that is, $k\times\infty$-matrices,
$\infty\times\ell$-matrices and $\infty\times\infty$-matrices) will be called
\textit{infinite matrices}. In contrast, $k\times\ell$-matrices with
$k,\ell\in\mathbb{N}$ will be called \textit{finite matrices}.

We have previously introduced the notation $A_{i,j}$ for the $\left(
i,j\right)  $-th entry of $A$ whenever $A$ is a $k\times\ell$-matrix. The same
notation will apply when $A$ is an infinite matrix (i.e., when one or both of
$k$ and $\ell$ is $\infty$).

If $u,v,w$ are three elements of $\mathbb{N}$, and if $A$ is a $u\times
v$-matrix, and if $B$ is a $v\times w$-matrix, then the product $AB$ is a
$u\times w$-matrix, and its entries are given by
\begin{align}
\left(  AB\right)  _{i,k}  &  =\sum_{j=0}^{v-1}A_{i,j}B_{j,k} \label{eq.ABik=}%
\\
&  \ \ \ \ \ \ \ \ \ \ \text{for all }i\in\left\{  0,1,\ldots,u-1\right\}
\text{ and }k\in\left\{  0,1,\ldots,w-1\right\}  .\nonumber
\end{align}
The same formula can be used to define $AB$ when some of $u,v,w$ are $\infty$
(keeping in mind that $\left\{  0,1,\ldots,\infty-1\right\}  =\mathbb{N}$),
but in this case it may fail to provide a well-defined result. Indeed, if
$v=\infty$, then the sum on the right hand side of (\ref{eq.ABik=}) is
infinite and thus may fail to be well-defined. Worse yet, even when products
of infinite matrices are well-defined, they can fail the associativity law
$\left(  AB\right)  C=A\left(  BC\right)  $. We shall not dwell on these
perversions, but rather restrict ourselves to a subclass of infinite matrices
which avoids them:

\begin{definition}
Let $u,v\in\mathbb{N}\cup\left\{  \infty\right\}  $. Let $A$ be a $u\times
v$-matrix. Let $k\in\mathbb{Z}$. We say that the matrix $A$ is $k$%
\textit{-lower-triangular} if and only if we have%
\[
\left(  A_{i,j}=0\ \ \ \ \ \ \ \ \ \ \text{for all }\left(  i,j\right)  \text{
satisfying }i<j+k\right)  .
\]

\end{definition}

\begin{definition}
A matrix $A$ is said to be \textit{quasi-lower-triangular} if and only if
there exists a $k\in\mathbb{Z}$ such that $A$ is $k$-lower-triangular.
\end{definition}

Note that we did not require our matrix $A$ to be square in these two
definitions. Unlike the standard kind of triangularity, our concept of
quasi-triangularity is meant to be a tameness condition, meant to guarantee
the well-definedness of an infinite sum; in particular, all finite matrices
are quasi-lower-triangular. Better yet, the following holds:\footnote{The
proofs of all propositions stated in Subsection \ref{subsect.form.inf-mat} are
left to the reader as easy exercises.}

\begin{proposition}
Let $k\in\mathbb{N}\cup\left\{  \infty\right\}  $ and $\ell\in\mathbb{N}$.
Then, any $k\times\ell$-matrix is quasi-lower-triangular. More concretely: Any
$k\times\ell$-matrix is $\left(  \ell-1\right)  $-lower-triangular.
\end{proposition}

\begin{proposition}
Let $A$ be a matrix (finite or infinite) such that all but finitely many
entries of $A$ are $0$. Then, $A$ is quasi-lower-triangular.
\end{proposition}

Quasi-lower-triangular matrices can be multiplied, as the following
proposition shows:

\begin{proposition}
Let $u,v,w\in\mathbb{N}\cup\left\{  \infty\right\}  $. Let $A$ be a
quasi-lower-triangular $u\times v$-matrix, and let $B$ be a
quasi-lower-triangular $v\times w$-matrix. Then, the product $AB$ is
well-defined (i.e. the infinite sum on the right hand side of (\ref{eq.ABik=})
is well-defined even if $v=\infty$) and is a quasi-lower-triangular $u\times
w$-matrix.

More concretely: If $k,\ell\in\mathbb{Z}$ are such that $A$ is $k$%
-lower-triangular and $B$ is $\ell$-lower-triangular, then $AB$ is $\left(
k+\ell\right)  $-lower-triangular.
\end{proposition}

Finally, multiplication of quasi-lower-triangular matrices is associative:

\begin{proposition}
Let $u,v,w,x\in\mathbb{N}\cup\left\{  \infty\right\}  $. Let $A$ be a
quasi-lower-triangular $u\times v$-matrix; let $B$ be a quasi-lower-triangular
$v\times w$-matrix; let $C$ be a quasi-lower-triangular $w\times x$-matrix.
Then, $\left(  AB\right)  C=A\left(  BC\right)  $.
\end{proposition}

This proposition entails that we can calculate with quasi-lower-triangular
matrices just as we can calculate with finite matrices. In particular, the
quasi-lower-triangular $\infty\times\infty$-matrices form a ring. Thus, a
quasi-lower-triangular $\infty\times\infty$-matrix has a well-defined $n$-th
power for each $n\in\mathbb{N}$.

\subsection{The matrices $S$ and $U_{b}$ and the vectors $H_{c}$ and $e_{j}$}

Let us now introduce several more players into the drama.

\subsubsection{Iverson brackets (truth values)}

We shall use the \textit{Iverson bracket notation}: If $\mathcal{A}$ is any
logical statement, then $\left[  \mathcal{A}\right]  $ will denote the integer
$%
\begin{cases}
1, & \text{if }\mathcal{A}\text{ is true;}\\
0, & \text{if }\mathcal{A}\text{ is false}%
\end{cases}
\in\left\{  0,1\right\}  $. This integer $\left[  \mathcal{A}\right]  $ is
called the \textit{truth value} of $\mathcal{A}$.

\subsubsection{$m$ and $a$}

We now return to our ring $\mathbb{L}$.

For the rest of this note, we fix an $m\in\mathbb{N}\cup\left\{
\infty\right\}  $ and an element $a\in\mathbb{L}$.

\subsubsection{The matrix $S$}

We define an $m\times m$-matrix $S\in\mathbb{L}^{m\times m}$ by%
\begin{equation}
S=\left(  \left[  j=i+1\right]  \right)  _{0\leq i<m,\ 0\leq j<m}.
\label{eq.S=}%
\end{equation}

This matrix $S$ looks as follows:

\begin{itemize}
\item If $m\in\mathbb{N}$, then%
\[
S=\left(
\begin{array}
[c]{ccccc}%
0 & 1 & 0 & \cdots & 0\\
0 & 0 & 1 & \cdots & 0\\
0 & 0 & 0 & \cdots & 0\\
\vdots & \vdots & \vdots & \ddots & \vdots\\
0 & 0 & 0 & \cdots & 0
\end{array}
\right)  .
\]

\item If $m=\infty$, then%
\[
S=\left(
\begin{array}
[c]{ccccc}%
0 & 1 & 0 & 0 & \cdots\\
0 & 0 & 1 & 0 & \cdots\\
0 & 0 & 0 & 1 & \cdots\\
0 & 0 & 0 & 0 & \cdots\\
\vdots & \vdots & \vdots & \vdots & \ddots
\end{array}
\right)  .
\]

\end{itemize}

The matrix $S$ (or, rather, the $\mathbb{L}$-linear map from $\mathbb{L}^{m}$
to $\mathbb{L}^{m}$ it represents\footnote{When $m=\infty$, you can read
$\mathbb{L}^{m}$ both as the direct sum $\bigoplus\limits_{i\in\mathbb{N}%
}\mathbb{L}$ and as the direct product $\prod_{i\in\mathbb{N}}\mathbb{L}$.
These are two different options, but either has an $\mathbb{L}$-linear map
represented by the matrix $S$.}) is often called the \textit{shift operator}.
Note that the matrix $S$ is quasi-lower-triangular\footnote{See Subsection
\ref{subsect.form.inf-mat} for the meaning of this word (and ignore it if you
don't care about the case of $m=\infty$).} (and, in fact, $\left(  -1\right)
$-lower-triangular\footnote{See Subsection \ref{subsect.form.inf-mat} for the
meaning of this word (and ignore it if you don't care about the case of
$m=\infty$).}), but of course not lower-triangular (unless $\mathbb{L}=0$ or
$m\leq1$).

\subsubsection{The matrix $U_{b}$}

If $n$ is a nonnegative integer, $T$ is a set and $f:T\rightarrow T$ is any
map, then $f^{n}$ will mean the composition $\underbrace{f\circ f\circ
\cdots\circ f}_{n\text{ times}}$; this is again a map from $T$ to $T$.

For any $b\in\mathbb{L}$, we define an $m\times m$-matrix $U_{b}\in
\mathbb{L}^{m\times m}$ by%
\begin{equation}
U_{b}=\left(
\begin{cases}
\dbinom{i}{j}\operatorname*{ad}\nolimits_{a}^{i-j}\left(  b\right)  , &
\text{if }i\geq j;\\
0, & \text{if }i<j
\end{cases}
\right)  _{0\leq i<m,\ 0\leq j<m}. \label{eq.Ub=}%
\end{equation}
(Here, of course, $\operatorname*{ad}\nolimits_{a}^{n}$ means $\left(
\operatorname*{ad}\nolimits_{a}\right)  ^{n}$ whenever $n\in\mathbb{N}$.)

This matrix $U_{b}$ looks as follows:

\begin{itemize}
\item If $b\in\mathbb{L}$ and $m\in\mathbb{N}$, then%
\[
U_{b}=\left(
\begin{array}
[c]{ccccc}%
b & 0 & 0 & \cdots & 0\\
\operatorname*{ad}\nolimits_{a}\left(  b\right)  & b & 0 & \cdots & 0\\
\operatorname*{ad}\nolimits_{a}^{2}\left(  b\right)  & 2\operatorname*{ad}%
\nolimits_{a}\left(  b\right)  & b & \cdots & 0\\
\vdots & \vdots & \vdots & \ddots & \vdots\\
\operatorname*{ad}\nolimits_{a}^{m-1}\left(  b\right)  & \left(  m-1\right)
\operatorname*{ad}\nolimits_{a}^{m-2}\left(  b\right)  & \dbinom{m-1}%
{2}\operatorname*{ad}\nolimits_{a}^{m-3}\left(  b\right)  & \cdots & b
\end{array}
\right)  .
\]

\item If $b\in\mathbb{L}$ and $m=\infty$, then%
\[
U_{b}=\left(
\begin{array}
[c]{ccccc}%
b & 0 & 0 & 0 & \cdots\\
\operatorname*{ad}\nolimits_{a}\left(  b\right)  & b & 0 & 0 & \cdots\\
\operatorname*{ad}\nolimits_{a}^{2}\left(  b\right)  & 2\operatorname*{ad}%
\nolimits_{a}\left(  b\right)  & b & 0 & \cdots\\
\operatorname*{ad}\nolimits_{a}^{3}\left(  b\right)  & 3\operatorname*{ad}%
\nolimits_{a}^{2}\left(  b\right)  & 3\operatorname*{ad}\nolimits_{a}\left(
b\right)  & b & \cdots\\
\vdots & \vdots & \vdots & \vdots & \ddots
\end{array}
\right)  .
\]

\end{itemize}

Note that the matrix $U_{b}$ is always lower-triangular and thus
quasi-lower-triangular\footnote{See Subsection \ref{subsect.form.inf-mat} for
the meaning of this word (and ignore it if you don't care about the case of
$m=\infty$).}.

\subsubsection{The column vector $H_{c}$}

Furthermore, for each $c\in\mathbb{L}$, we define an $m\times1$-matrix
$H_{c}\in\mathbb{L}^{m\times1}$ by%
\begin{equation}
H_{c}=\left(  a^{i}c\right)  _{0\leq i<m,\ 0\leq j<1}. \label{eq.Hc=}%
\end{equation}
Thus, $H_{c}$ is an $m\times1$-matrix, i.e., a column vector of size $m$. It
looks as follows:

\begin{itemize}
\item If $c\in\mathbb{L}$ and $m\in\mathbb{N}$, then%
\[
H_{c}=\left(
\begin{array}
[c]{c}%
a^{0}c\\
a^{1}c\\
\vdots\\
a^{m-1}c
\end{array}
\right)  .
\]

\item If $c\in\mathbb{L}$ and $m=\infty$, then%
\[
H_{c}=\left(
\begin{array}
[c]{c}%
a^{0}c\\
a^{1}c\\
a^{2}c\\
\vdots
\end{array}
\right)  .
\]

\end{itemize}

Clearly, the matrix $H_{c}$ is quasi-lower-triangular\footnote{See Subsection
\ref{subsect.form.inf-mat} for the meaning of this word (and ignore it if you
don't care about the case of $m=\infty$).}, since it has only one column.

\subsubsection{The column vector $e_{j}$}

For each integer $j$ with $0\leq j<m$, we let $e_{j}\in\mathbb{L}^{m\times1}$
be the $m\times1$-matrix defined by%
\begin{equation}
e_{j}=\left(  \left[  p=j\right]  \right)  _{0\leq p<m,\ 0\leq q<1}.
\label{eq.ej=}%
\end{equation}
In other words, $e_{j}$ is the column vector (of size $m$) whose $j$-th entry
is $1$ and whose all other entries are $0$. This column vector $e_{j}$ is
commonly known as the $j$-th \textit{standard basis vector} of $\mathbb{L}%
^{m\times1}$.

Thus, in particular, $e_{0}$ is a column vector with a $1$ in its topmost
position and $0$'s everywhere else. It looks as follows:

\begin{itemize}
\item If $m\in\mathbb{N}$, then%
\[
e_{0}=\left(
\begin{array}
[c]{c}%
1\\
0\\
0\\
\vdots\\
0
\end{array}
\right)  .
\]

\item If $m=\infty$, then%
\[
e_{0}=\left(
\begin{array}
[c]{c}%
1\\
0\\
0\\
0\\
\vdots
\end{array}
\right)  .
\]

\end{itemize}

Thus, $e_{0}^{T}$ is a row vector with a $1$ in its leftmost position and
$0$'s everywhere else. This shows that the matrix $e_{0}^{T}$ is
quasi-lower-triangular\footnote{See Subsection \ref{subsect.form.inf-mat} for
the meaning of this word (and ignore it if you don't care about the case of
$m=\infty$).}.

\subsection{\label{subsect.form.thm}The general formula}

We are now ready to state our main claim:

\begin{theorem}
\label{thm.gen}Let $n\in\mathbb{N}$ be such that $n<m$. Let $b\in\mathbb{L}$.
Then,%
\[
\left(  ba\right)  ^{n}=e_{0}^{T}\left(  U_{b}S\right)  ^{n}H_{1}.
\]

\end{theorem}

(The right hand side of this equality is a $1\times1$-matrix, while the left
hand side is an element of $\mathbb{L}$. The equality thus makes sense because
we are equating $1\times1$-matrices with elements of $\mathbb{L}$.)

\begin{example}
Let us set $m=3$ and $n=2$ in Theorem \ref{thm.gen}. Then, Theorem
\ref{thm.gen} claims that $\left(  ba\right)  ^{2}=e_{0}^{T}\left(
U_{b}S\right)  ^{2}H_{1}$. Let us check this: We have%
\[
U_{b}=\left(
\begin{array}
[c]{ccc}%
b & 0 & 0\\
\operatorname*{ad}\nolimits_{a}\left(  b\right)  & b & 0\\
\operatorname*{ad}\nolimits_{a}^{2}\left(  b\right)  & 2\operatorname*{ad}%
\nolimits_{a}\left(  b\right)  & b
\end{array}
\right)  \ \ \ \ \ \ \ \ \ \ \text{and}\ \ \ \ \ \ \ \ \ \ S=\left(
\begin{array}
[c]{ccc}%
0 & 1 & 0\\
0 & 0 & 1\\
0 & 0 & 0
\end{array}
\right)  ,
\]
so that%
\[
U_{b}S=\left(
\begin{array}
[c]{ccc}%
b & 0 & 0\\
\operatorname*{ad}\nolimits_{a}\left(  b\right)  & b & 0\\
\operatorname*{ad}\nolimits_{a}^{2}\left(  b\right)  & 2\operatorname*{ad}%
\nolimits_{a}\left(  b\right)  & b
\end{array}
\right)  \left(
\begin{array}
[c]{ccc}%
0 & 1 & 0\\
0 & 0 & 1\\
0 & 0 & 0
\end{array}
\right)  =\left(
\begin{array}
[c]{ccc}%
0 & b & 0\\
0 & \operatorname*{ad}\nolimits_{a}\left(  b\right)  & b\\
0 & \operatorname*{ad}\nolimits_{a}^{2}\left(  b\right)  & 2\operatorname*{ad}%
\nolimits_{a}\left(  b\right)
\end{array}
\right)
\]
and therefore%
\begin{align*}
\left(  U_{b}S\right)  ^{2}  &  =\left(
\begin{array}
[c]{ccc}%
0 & b & 0\\
0 & \operatorname*{ad}\nolimits_{a}\left(  b\right)  & b\\
0 & \operatorname*{ad}\nolimits_{a}^{2}\left(  b\right)  & 2\operatorname*{ad}%
\nolimits_{a}\left(  b\right)
\end{array}
\right)  ^{2}\\
&  =\left(
\begin{array}
[c]{ccc}%
0 & b\operatorname*{ad}\nolimits_{a}\left(  b\right)  & b^{2}\\
0 & \left(  \operatorname*{ad}\nolimits_{a}\left(  b\right)  \right)
^{2}+b\operatorname*{ad}\nolimits_{a}^{2}\left(  b\right)  &
3b\operatorname*{ad}\nolimits_{a}\left(  b\right) \\
0 & 3\operatorname*{ad}\nolimits_{a}\left(  b\right)  \operatorname*{ad}%
\nolimits_{a}^{2}\left(  b\right)  & 4\left(  \operatorname*{ad}%
\nolimits_{a}\left(  b\right)  \right)  ^{2}+b\operatorname*{ad}%
\nolimits_{a}^{2}\left(  b\right)
\end{array}
\right)  .
\end{align*}
Multiplying $e_{0}^{T}=\left(
\begin{array}
[c]{ccc}%
1 & 0 & 0
\end{array}
\right)  $ by this equality, we find%
\begin{align*}
e_{0}^{T}\left(  U_{b}S\right)  ^{2}  &  =\left(
\begin{array}
[c]{ccc}%
1 & 0 & 0
\end{array}
\right)  \left(
\begin{array}
[c]{ccc}%
0 & b\operatorname*{ad}\nolimits_{a}\left(  b\right)  & b^{2}\\
0 & \left(  \operatorname*{ad}\nolimits_{a}\left(  b\right)  \right)
^{2}+b\operatorname*{ad}\nolimits_{a}^{2}\left(  b\right)  &
3b\operatorname*{ad}\nolimits_{a}\left(  b\right) \\
0 & 3\operatorname*{ad}\nolimits_{a}\left(  b\right)  \operatorname*{ad}%
\nolimits_{a}^{2}\left(  b\right)  & 4\left(  \operatorname*{ad}%
\nolimits_{a}\left(  b\right)  \right)  ^{2}+b\operatorname*{ad}%
\nolimits_{a}^{2}\left(  b\right)
\end{array}
\right) \\
&  =\left(
\begin{array}
[c]{ccc}%
0 & b\operatorname*{ad}\nolimits_{a}\left(  b\right)  & b^{2}%
\end{array}
\right)  .
\end{align*}
Multiplying this equality by $H_{1}=\left(
\begin{array}
[c]{c}%
a^{0}1\\
a^{1}1\\
a^{2}1
\end{array}
\right)  =\left(
\begin{array}
[c]{c}%
a^{0}\\
a^{1}\\
a^{2}%
\end{array}
\right)  $, we obtain%
\begin{align*}
e_{0}^{T}\left(  U_{b}S\right)  ^{2}H_{1}  &  =\left(
\begin{array}
[c]{ccc}%
0 & b\operatorname*{ad}\nolimits_{a}\left(  b\right)  & b^{2}%
\end{array}
\right)  \left(
\begin{array}
[c]{c}%
a^{0}\\
a^{1}\\
a^{2}%
\end{array}
\right)  =0a^{0}+b\operatorname*{ad}\nolimits_{a}\left(  b\right)  a^{1}%
+b^{2}a^{2}\\
&  =b\underbrace{\operatorname*{ad}\nolimits_{a}\left(  b\right)
}_{\substack{=\left[  a,b\right]  \\\text{(by the definition of }%
\operatorname*{ad}\nolimits_{a}\text{)}}}a+b^{2}a^{2}=b\underbrace{\left[
a,b\right]  }_{=ab-ba}a+b^{2}a^{2}\\
&  =b\left(  ab-ba\right)  a+b^{2}a^{2}=baba-bbaa+bbaa=baba=\left(  ba\right)
^{2}.
\end{align*}
This confirms the claim that $\left(  ba\right)  ^{2}=e_{0}^{T}\left(
U_{b}S\right)  ^{2}H_{1}$.
\end{example}

\section{The proof}

\subsection{The idea}

Proving Theorem \ref{thm.gen} is not hard, but it will take us some
preparation due to the bookkeeping required. The main idea manifests itself in
its cleanest form when $m=\infty$; indeed, it is not hard to prove the
following two facts:\footnote{We shall prove these two facts later.}

\begin{proposition}
\label{prop.SHinf}Assume that $m=\infty$. Let $c\in\mathbb{L}$. Then,
$SH_{c}=H_{ac}$.
\end{proposition}

\begin{proposition}
\label{prop.UHinf}Let $b\in\mathbb{L}$ and $c\in\mathbb{L}$. Then, $U_{b}%
H_{c}=H_{bc}$.
\end{proposition}

If $m=\infty$, then we can use Proposition \ref{prop.SHinf} and Proposition
\ref{prop.UHinf} to conclude that $\left(  U_{b}S\right)  H_{c}=H_{bac}$ for
each $b\in\mathbb{L}$ and $c\in\mathbb{L}$. Thus, by induction, we can
conclude that $\left(  U_{b}S\right)  ^{n}H_{c}=H_{\left(  ba\right)  ^{n}c}$
for each $n\in\mathbb{N}$, $b\in\mathbb{L}$ and $c\in\mathbb{L}$ (as long as
$m=\infty$). Applying this to $c=1$ and multiplying the resulting equality by
$e_{0}^{T}$ on both sides, we then obtain $e_{0}^{T}\left(  U_{b}S\right)
^{n}H_{1}=e_{0}^{T}H_{\left(  ba\right)  ^{n}1}=\left(  ba\right)  ^{n}$ (the
last equality sign is easy). This proves Theorem \ref{thm.gen} in the case
when $m=\infty$.

Unfortunately, this argument breaks down if $m\in\mathbb{N}$. In fact,
Proposition \ref{prop.SHinf} is true only for $m=\infty$; otherwise, the
vectors $SH_{c}$ and $H_{ac}$ differ in their last entry. This
\textquotedblleft corruption\textquotedblright\ then spreads further to
earlier and earlier entries as we inductively multiply by $U_{b}$ and by $S$.
What saves us is that it only spreads one entry at a time when we multiply by
$S$, and does not spread at all when we multiply by $U_{b}$; thus it does not
reach the first (i.e., $0$-th) entry as long as we multiply by $U_{b}S$ only
$n$ times. But this needs to be formalized and proved. This is what we shall
be doing further below.

\subsection{A lemma about $\operatorname*{ad}\nolimits_{a}$}

Before we come to this, however, we need a basic lemma about commutators:

\begin{lemma}
\label{lem.ada1}Let $b\in\mathbb{L}$ and $i\in\mathbb{N}$. Then,%
\[
a^{i}b=\sum_{j=0}^{i}\dbinom{i}{j}\operatorname*{ad}\nolimits_{a}^{i-j}\left(
b\right)  \cdot a^{j}.
\]

\end{lemma}

It is not hard to prove Lemma \ref{lem.ada1} by induction on $i$. However,
there is a slicker proof. It relies on the following well-known fact:

\begin{proposition}
\label{prop.binom}Let $\mathbb{A}$ be a ring. Let $x$ and $y$ be two elements
of $\mathbb{A}$ such that $xy=yx$. Then,%
\[
\left(  x+y\right)  ^{n}=\sum_{k=0}^{n}\dbinom{n}{k}x^{k}y^{n-k}%
\ \ \ \ \ \ \ \ \ \ \text{for every }n\in\mathbb{N}\text{.}%
\]

\end{proposition}

Proposition \ref{prop.binom} is a straightforward generalization of the
binomial formula to two commuting elements of an arbitrary ring.

\begin{proof}
[Proof of Lemma \ref{lem.ada1}.]Let $\operatorname*{End}\mathbb{L}$ denote the
endomorphism ring of the $\mathbb{Z}$-module $\mathbb{L}$. Thus, the elements
of $\operatorname*{End}\mathbb{L}$ are the $\mathbb{Z}$-linear maps from
$\mathbb{L}$ to $\mathbb{L}$.

Define the map $L_{a}:\mathbb{L}\rightarrow\mathbb{L}$ by%
\[
\left(  L_{a}\left(  c\right)  =ac\ \ \ \ \ \ \ \ \ \ \text{for all }%
c\in\mathbb{L}\right)  .
\]
Clearly, this map $L_{a}$ is $\mathbb{Z}$-linear; thus, it belongs to
$\operatorname*{End}\mathbb{L}$.

Define the map $R_{a}:\mathbb{L}\rightarrow\mathbb{L}$ by%
\[
\left(  R_{a}\left(  c\right)  =ca\ \ \ \ \ \ \ \ \ \ \text{for all }%
c\in\mathbb{L}\right)  .
\]
Clearly, this map $R_{a}$ is $\mathbb{Z}$-linear; thus, it belongs to
$\operatorname*{End}\mathbb{L}$.

We have $\operatorname*{ad}\nolimits_{a}=L_{a}-R_{a}$%
\ \ \ \ \footnote{\textit{Proof.} Let $c\in\mathbb{L}$. Then, $L_{a}\left(
c\right)  =ac$ (by the definition of $L_{a}$) and $R_{a}\left(  c\right)  =ca$
(by the definition of $R_{a}$). Hence,
\[
\left(  L_{a}-R_{a}\right)  \left(  c\right)  =\underbrace{L_{a}\left(
c\right)  }_{=ac}-\underbrace{R_{a}\left(  c\right)  }_{=ca}=ac-ca.
\]
Comparing this with%
\begin{align*}
\operatorname*{ad}\nolimits_{a}\left(  c\right)   &  =\left[  a,c\right]
\ \ \ \ \ \ \ \ \ \ \left(  \text{by the definition of }\operatorname*{ad}%
\nolimits_{a}\right) \\
&  =ac-ca\ \ \ \ \ \ \ \ \ \ \left(  \text{by the definition of }\left[
a,c\right]  \right)  ,
\end{align*}
we obtain $\operatorname*{ad}\nolimits_{a}\left(  c\right)  =\left(
L_{a}-R_{a}\right)  \left(  c\right)  $.
\par
Now, forget that we fixed $c$. We thus have shown that $\operatorname*{ad}%
\nolimits_{a}\left(  c\right)  =\left(  L_{a}-R_{a}\right)  \left(  c\right)
$ for each $c\in\mathbb{L}$. In other words, $\operatorname*{ad}%
\nolimits_{a}=L_{a}-R_{a}$. Qed.}. Hence, $\operatorname*{ad}\nolimits_{a}$
belongs to $\operatorname*{End}\mathbb{L}$ (since $L_{a}$ and $R_{a}$ belong
to $\operatorname*{End}\mathbb{L}$). Also, $R_{a}+\operatorname*{ad}%
\nolimits_{a}=L_{a}$ (since $\operatorname*{ad}\nolimits_{a}=L_{a}-R_{a}$).

Furthermore, the elements $L_{a}$ and $R_{a}$ of $\operatorname*{End}%
\mathbb{L}$ satisfy $R_{a}\circ L_{a}=L_{a}\circ R_{a}$%
\ \ \ \ \footnote{\textit{Proof.} Let $c\in\mathbb{L}$. The definition of
$L_{a}$ yields $L_{a}\left(  c\right)  =ac$ and $L_{a}\left(  R_{a}\left(
c\right)  \right)  =a\cdot R_{a}\left(  c\right)  $. The definition of $R_{a}$
yields $R_{a}\left(  c\right)  =ca$ and $R_{a}\left(  L_{a}\left(  c\right)
\right)  =L_{a}\left(  c\right)  \cdot a$. Now, comparing%
\[
\left(  L_{a}\circ R_{a}\right)  \left(  c\right)  =L_{a}\left(  R_{a}\left(
c\right)  \right)  =a\cdot\underbrace{R_{a}\left(  c\right)  }_{=ca}=a\cdot
ca=aca
\]
with%
\[
\left(  R_{a}\circ L_{a}\right)  \left(  c\right)  =R_{a}\left(  L_{a}\left(
c\right)  \right)  =\underbrace{L_{a}\left(  c\right)  }_{=ac}\cdot a=ac\cdot
a=aca,
\]
we obtain $\left(  R_{a}\circ L_{a}\right)  \left(  c\right)  =\left(
L_{a}\circ R_{a}\right)  \left(  c\right)  $.
\par
Forget that we fixed $c$. We thus have proven that $\left(  R_{a}\circ
L_{a}\right)  \left(  c\right)  =\left(  L_{a}\circ R_{a}\right)  \left(
c\right)  $ for each $c\in\mathbb{L}$. In other words, $R_{a}\circ L_{a}%
=L_{a}\circ R_{a}$.}. But $\operatorname*{End}\mathbb{L}$ is a ring with
multiplication $\circ$; thus, in particular, the operation $\circ$ is
distributive (over $+$) on $\operatorname*{End}\mathbb{L}$. Since $L_{a}$,
$R_{a}$ and $\operatorname*{ad}\nolimits_{a}$ belong to $\operatorname*{End}%
\mathbb{L}$, we thus have%
\begin{align*}
R_{a}\circ\underbrace{\operatorname*{ad}\nolimits_{a}}_{=L_{a}-R_{a}}  &
=R_{a}\circ\left(  L_{a}-R_{a}\right)  =\underbrace{R_{a}\circ L_{a}}%
_{=L_{a}\circ R_{a}}-R_{a}\circ R_{a}\\
&  =L_{a}\circ R_{a}-R_{a}\circ R_{a}=\underbrace{\left(  L_{a}-R_{a}\right)
}_{=\operatorname*{ad}\nolimits_{a}}\circ R_{a}=\operatorname*{ad}%
\nolimits_{a}\circ R_{a}.
\end{align*}
Hence, Proposition \ref{prop.binom} (applied to $\mathbb{A}%
=\operatorname*{End}\mathbb{L}$, $x=R_{a}$, $y=\operatorname*{ad}%
\nolimits_{a}$ and $n=i$) yields%
\[
\left(  R_{a}+\operatorname*{ad}\nolimits_{a}\right)  ^{i}=\sum_{k=0}%
^{i}\dbinom{i}{k}R_{a}^{k}\circ\operatorname*{ad}\nolimits_{a}^{i-k}%
=\sum_{j=0}^{i}\dbinom{i}{j}R_{a}^{j}\circ\operatorname*{ad}\nolimits_{a}%
^{i-j}%
\]
(here, we have renamed the index $k$ as $j$ in the sum). In view of
$R_{a}+\operatorname*{ad}\nolimits_{a}=L_{a}$, this rewrites as%
\begin{equation}
L_{a}^{i}=\sum_{j=0}^{i}\dbinom{i}{j}R_{a}^{j}\circ\operatorname*{ad}%
\nolimits_{a}^{i-j}. \label{pf.lem.ada1.5}%
\end{equation}

But each $k\in\mathbb{N}$ satisfies%
\begin{equation}
L_{a}^{k}\left(  c\right)  =a^{k}c\ \ \ \ \ \ \ \ \ \ \text{for each }%
c\in\mathbb{L}. \label{pf.lem.ada1.Lakc}%
\end{equation}

\begin{vershort}
[\textit{Proof of (\ref{pf.lem.ada1.Lakc}):} It is straightforward to prove
(\ref{pf.lem.ada1.Lakc}) by induction on $k$.]
\end{vershort}

\begin{verlong}
[\textit{Proof of (\ref{pf.lem.ada1.Lakc}):} We shall prove
(\ref{pf.lem.ada1.Lakc}) by induction on $k$:

\textit{Induction base:} We have $\underbrace{L_{a}^{0}}_{=\operatorname*{id}%
}\left(  c\right)  =\operatorname*{id}\left(  c\right)  =c=a^{0}c$ (since
$\underbrace{a^{0}}_{=1}c=c$) for each $c\in\mathbb{L}$. In other words,
(\ref{pf.lem.ada1.Lakc}) holds for $k=0$. This completes the induction base.

\textit{Induction step:} Let $g\in\mathbb{N}$. Assume that
(\ref{pf.lem.ada1.Lakc}) holds for $k=g$. We must prove that
(\ref{pf.lem.ada1.Lakc}) holds for $k=g+1$.

We have assumed that (\ref{pf.lem.ada1.Lakc}) holds for $k=g$. In other words,
we have%
\begin{equation}
L_{a}^{g}\left(  c\right)  =a^{g}c\ \ \ \ \ \ \ \ \ \ \text{for each }%
c\in\mathbb{L}. \label{pf.lem.ada1.Lakc.pf.IH}%
\end{equation}

Now, for each $c\in\mathbb{L}$, we have
\begin{align*}
\underbrace{L_{a}^{g+1}}_{=L_{a}\circ L_{a}^{g}}\left(  c\right)   &  =\left(
L_{a}\circ L_{a}^{g}\right)  \left(  c\right)  =L_{a}\left(  L_{a}^{g}\left(
c\right)  \right)  =a\cdot\underbrace{L_{a}^{g}\left(  c\right)
}_{\substack{=a^{g}c\\\text{(by (\ref{pf.lem.ada1.Lakc.pf.IH}))}%
}}\ \ \ \ \ \ \ \ \ \ \left(  \text{by the definition of }L_{a}\right) \\
&  =\underbrace{a\cdot a^{g}}_{=a^{g+1}}c=a^{g+1}c.
\end{align*}

Thus, we have shown that $L_{a}^{g+1}\left(  c\right)  =a^{g+1}c$ for each
$c\in\mathbb{L}$. In other words, (\ref{pf.lem.ada1.Lakc}) holds for $k=g+1$.
This completes the induction step. Thus, the proof of (\ref{pf.lem.ada1.Lakc})
is complete.]
\end{verlong}

Furthermore, each $k\in\mathbb{N}$ satisfies%
\begin{equation}
R_{a}^{k}\left(  c\right)  =ca^{k}\ \ \ \ \ \ \ \ \ \ \text{for each }%
c\in\mathbb{L}. \label{pf.lem.ada1.Rakc}%
\end{equation}

\begin{vershort}
[\textit{Proof of (\ref{pf.lem.ada1.Rakc}):} It is straightforward to prove
(\ref{pf.lem.ada1.Rakc}) by induction on $k$.]
\end{vershort}

\begin{verlong}
[\textit{Proof of (\ref{pf.lem.ada1.Rakc}):} We shall prove
(\ref{pf.lem.ada1.Rakc}) by induction on $k$:

\textit{Induction base:} We have $\underbrace{R_{a}^{0}}_{=\operatorname*{id}%
}\left(  c\right)  =\operatorname*{id}\left(  c\right)  =c=ca^{0}$ (since
$c\underbrace{a^{0}}_{=1}=c$) for each $c\in\mathbb{L}$. In other words,
(\ref{pf.lem.ada1.Rakc}) holds for $k=0$. This completes the induction base.

\textit{Induction step:} Let $g\in\mathbb{N}$. Assume that
(\ref{pf.lem.ada1.Rakc}) holds for $k=g$. We must prove that
(\ref{pf.lem.ada1.Rakc}) holds for $k=g+1$.

We have assumed that (\ref{pf.lem.ada1.Rakc}) holds for $k=g$. In other words,
we have%
\begin{equation}
R_{a}^{g}\left(  c\right)  =ca^{g}\ \ \ \ \ \ \ \ \ \ \text{for each }%
c\in\mathbb{L}. \label{pf.lem.ada1.Rakc.pf.IH}%
\end{equation}

Now, for each $c\in\mathbb{L}$, we have
\begin{align*}
\underbrace{R_{a}^{g+1}}_{=R_{a}\circ R_{a}^{g}}\left(  c\right)   &  =\left(
R_{a}\circ R_{a}^{g}\right)  \left(  c\right)  =R_{a}\left(  R_{a}^{g}\left(
c\right)  \right)  =\underbrace{R_{a}^{g}\left(  c\right)  }%
_{\substack{=ca^{g}\\\text{(by (\ref{pf.lem.ada1.Rakc.pf.IH}))}}}\cdot
a\ \ \ \ \ \ \ \ \ \ \left(  \text{by the definition of }R_{a}\right) \\
&  =c\underbrace{a^{g}\cdot a}_{=a^{g+1}}=ca^{g+1}.
\end{align*}

Thus, we have shown that $R_{a}^{g+1}\left(  c\right)  =ca^{g+1}$ for each
$c\in\mathbb{L}$. In other words, (\ref{pf.lem.ada1.Rakc}) holds for $k=g+1$.
This completes the induction step. Thus, the proof of (\ref{pf.lem.ada1.Rakc})
is complete.]
\end{verlong}

Now, applying both sides of the equality (\ref{pf.lem.ada1.5}) to $b$, we
obtain%
\begin{align*}
L_{a}^{i}\left(  b\right)   &  =\left(  \sum_{j=0}^{i}\dbinom{i}{j}R_{a}%
^{j}\circ\operatorname*{ad}\nolimits_{a}^{i-j}\right)  \left(  b\right)
=\sum_{j=0}^{i}\dbinom{i}{j}\underbrace{\left(  R_{a}^{j}\circ
\operatorname*{ad}\nolimits_{a}^{i-j}\right)  \left(  b\right)  }%
_{\substack{=R_{a}^{j}\left(  \operatorname*{ad}\nolimits_{a}^{i-j}\left(
b\right)  \right)  \\=\operatorname*{ad}\nolimits_{a}^{i-j}\left(  b\right)
\cdot a^{j}\\\text{(by (\ref{pf.lem.ada1.Rakc}), applied}\\\text{to }k=j\text{
and }c=\operatorname*{ad}\nolimits_{a}^{i-j}\left(  b\right)  \text{)}}}\\
&  =\sum_{j=0}^{i}\dbinom{i}{j}\operatorname*{ad}\nolimits_{a}^{i-j}\left(
b\right)  \cdot a^{j}.
\end{align*}
Comparing this with%
\[
L_{a}^{i}\left(  b\right)  =a^{i}b\ \ \ \ \ \ \ \ \ \ \left(  \text{by
(\ref{pf.lem.ada1.Lakc}), applied to }k=i\text{ and }c=b\right)  ,
\]
we obtain
\[
a^{i}b=\sum_{j=0}^{i}\dbinom{i}{j}\operatorname*{ad}\nolimits_{a}^{i-j}\left(
b\right)  \cdot a^{j}.
\]
This proves Lemma \ref{lem.ada1}.
\end{proof}

\subsection{Formulas for $e_{i}^{T}A$}

We next recall a simple property of the vectors $e_{i}$:

\begin{lemma}
\label{lem.ei.row}Let $\ell\in\mathbb{N}\cup\left\{  \infty\right\}  $ and
$i\in\mathbb{N}$ be such that $0\leq i<m$. Let $A$ be an $m\times\ell$-matrix.
Then,%
\[
e_{i}^{T}A=\left(  \text{the }i\text{-th row of }A\right)  .
\]

\end{lemma}

Note that the product $e_{i}^{T}A$ on the left hand side of Lemma
\ref{lem.ei.row} is always well-defined, even when $\ell$ and $m$ are $\infty
$. (This stems from the fact that the row vector $e_{i}^{T}$ has only one
nonzero entry.)

Lemma \ref{lem.ei.row} says that the $i$-th row of $A$ can be extracted by
multiplying $A$ from the left by the row vector $e_{i}^{T}=\left(
\begin{array}
[c]{ccccccccc}%
0 & 0 & \cdots & 0 & 1 & 0 & 0 & \cdots & 0
\end{array}
\right)  $ (here, the $1$ is at the $i$-th position). This is a known fact
from linear algebra and is easy to prove.

\begin{verlong}
\begin{proof}
[Proof of Lemma \ref{lem.ei.row}.]Applying (\ref{eq.ej=}) to $j=i$, we find%
\[
e_{i}=\left(  \left[  p=i\right]  \right)  _{0\leq p<m,\ 0\leq q<1}.
\]
Thus, by the definition of the transpose of a matrix, we obtain%
\[
e_{i}^{T}=\left(  \left[  q=i\right]  \right)  _{0\leq p<1,\ 0\leq q<m}.
\]
Hence,%
\begin{equation}
\left(  e_{i}^{T}\right)  _{p,q}=\left[  q=i\right]
\label{pf.lem.ei.row.eiTpq=}%
\end{equation}
for any integers $p$ and $q$ satisfying $0\leq p<1$ and $0\leq q<m$.

Note that $i\in\left\{  0,1,\ldots,m-1\right\}  $ (since $0\leq i<m$). (Recall
that $\left\{  0,1,\ldots,\infty-1\right\}  =\mathbb{N}$.)

Now, for any integers $p$ and $q$ satisfying $0\leq p<1$ and $0\leq q<\ell$,
we have%
\begin{align}
\left(  e_{i}^{T}A\right)  _{p,q}  &  =\underbrace{\sum_{j=0}^{m-1}}%
_{=\sum_{j\in\left\{  0,1,\ldots,m-1\right\}  }}\underbrace{\left(  e_{i}%
^{T}\right)  _{p,j}}_{\substack{=\left[  j=i\right]  \\\text{(by
(\ref{pf.lem.ei.row.eiTpq=}),}\\\text{applied to }j\\\text{instead of
}q\text{)}}}A_{j,q}\nonumber\\
&  \ \ \ \ \ \ \ \ \ \ \left(
\begin{array}
[c]{c}%
\text{by (\ref{eq.ABik=}), applied to }1\text{, }m\text{, }\ell\text{, }%
e_{i}^{T}\text{, }A\text{, }p\text{ and }q\\
\text{instead of }u\text{, }v\text{, }w\text{, }A\text{, }B\text{, }i\text{
and }k
\end{array}
\right) \nonumber\\
&  =\sum_{j\in\left\{  0,1,\ldots,m-1\right\}  }\left[  j=i\right]
A_{j,q}=\underbrace{\left[  i=i\right]  }_{\substack{=1\\\text{(since
}i=i\text{)}}}A_{i,q}+\sum_{\substack{j\in\left\{  0,1,\ldots,m-1\right\}
;\\j\neq i}}\underbrace{\left[  j=i\right]  }_{\substack{=0\\\text{(since
}j\neq i\text{)}}}A_{j,q}\nonumber\\
&  \ \ \ \ \ \ \ \ \ \ \left(
\begin{array}
[c]{c}%
\text{here, we have split off the addend for }j=i\\
\text{from the sum (since }i\in\left\{  0,1,\ldots,m-1\right\}  \text{)}%
\end{array}
\right) \nonumber\\
&  =A_{i,q}+\underbrace{\sum_{\substack{j\in\left\{  0,1,\ldots,m-1\right\}
;\\j\neq i}}0A_{j,q}}_{=0}=A_{i,q}. \label{pf.lem.ei.row.4}%
\end{align}

But $e_{i}^{T}A$ is a $1\times\ell$-matrix. Thus,%
\begin{align*}
e_{i}^{T}A  &  =\left(  \left(  e_{i}^{T}A\right)  _{p,q}\right)  _{0\leq
p<1,\ 0\leq q<\ell}=\left(  A_{i,q}\right)  _{0\leq p<1,\ 0\leq q<\ell
}\ \ \ \ \ \ \ \ \ \ \left(  \text{by (\ref{pf.lem.ei.row.4})}\right) \\
&  =\left(
\begin{array}
[c]{cccc}%
A_{i,0} & A_{i,1} & \cdots & A_{i,\ell-1}%
\end{array}
\right)  =\left(  \text{the }i\text{-th row of }A\right)
\end{align*}
(since $A=\left(  A_{i,j}\right)  _{0\leq i<m,\ 0\leq j<\ell}$). This proves
Lemma \ref{lem.ei.row}.
\end{proof}
\end{verlong}

The next lemma is a slight restatement of Lemma \ref{lem.ei.row} in the case
when $\ell=m$:

\begin{lemma}
\label{lem.ei.row2}Let $i\in\mathbb{N}$ be such that $0\leq i<m$. Let $A$ be
an $m\times m$-matrix. Then,%
\[
e_{i}^{T}A=\sum_{j=0}^{m-1}A_{i,j}e_{j}^{T}.
\]

\end{lemma}

\begin{vershort}
\begin{proof}
[Proof of Lemma \ref{lem.ei.row2}.]For each $j\in\left\{  0,1,\ldots
,m-1\right\}  $, we have $e_{j}=\left(  \left[  p=j\right]  \right)  _{0\leq
p<m,\ 0\leq q<1}$ (by (\ref{eq.ej=})) and thus
\[
e_{j}^{T}=\left(  \left[  q=j\right]  \right)  _{0\leq p<1,\ 0\leq
q<m}\ \ \ \ \ \ \ \ \ \ \left(  \text{by the definition of the transpose of a
matrix}\right)  .
\]
Hence,%
\begin{align}
\sum_{j=0}^{m-1}A_{i,j}\underbrace{e_{j}^{T}}_{=\left(  \left[  q=j\right]
\right)  _{0\leq p<1,\ 0\leq q<m}}  &  =\sum_{j=0}^{m-1}A_{i,j}\left(  \left[
q=j\right]  \right)  _{0\leq p<1,\ 0\leq q<m}\nonumber\\
&  =\left(  \sum_{j=0}^{m-1}A_{i,j}\left[  q=j\right]  \right)  _{0\leq
p<1,\ 0\leq q<m}. \label{pf.lem.ei.row2.short.1}%
\end{align}
But for each $q\in\left\{  0,1,\ldots,m-1\right\}  $, we have%
\begin{align*}
\sum_{j=0}^{m-1}A_{i,j}\left[  q=j\right]   &  =A_{i,q}\underbrace{\left[
q=q\right]  }_{\substack{=1\\\text{(since }q=q\text{)}}}+\sum_{\substack{j\in
\left\{  0,1,\ldots,m-1\right\}  ;\\j\neq q}}A_{i,j}\underbrace{\left[
q=j\right]  }_{\substack{=0\\\text{(because }j\neq q\text{)}}}\\
&  \ \ \ \ \ \ \ \ \ \ \left(
\begin{array}
[c]{c}%
\text{here, we have split off the addend for }j=q\\
\text{from the sum (since }q\in\left\{  0,1,\ldots,m-1\right\}  \text{)}%
\end{array}
\right) \\
&  =A_{i,q}+\underbrace{\sum_{\substack{j\in\left\{  0,1,\ldots,m-1\right\}
;\\j\neq q}}A_{i,j}0}_{=0}=A_{i,q}.
\end{align*}
Hence,
\[
\left(  \sum_{j=0}^{m-1}A_{i,j}\left[  q=j\right]  \right)  _{0\leq
p<1,\ 0\leq q<m}=\left(  A_{i,q}\right)  _{0\leq p<1,\ 0\leq q<m}=\left(
\text{the }i\text{-th row of }A\right)
\]
(since $A=\left(  A_{i,j}\right)  _{0\leq i<m,\ 0\leq j<m}$). Hence,
(\ref{pf.lem.ei.row2.short.1}) becomes%
\[
\sum_{j=0}^{m-1}A_{i,j}e_{j}^{T}=\left(  \sum_{j=0}^{m-1}A_{i,j}\left[
q=j\right]  \right)  _{0\leq p<1,\ 0\leq q<m}=\left(  \text{the }i\text{-th
row of }A\right)  =e_{i}^{T}A
\]
(since Lemma \ref{lem.ei.row} yields $e_{i}^{T}A=\left(  \text{the }i\text{-th
row of }A\right)  $). This proves Lemma \ref{lem.ei.row2}.
\end{proof}
\end{vershort}

\begin{verlong}
\begin{proof}
[Proof of Lemma \ref{lem.ei.row2}.]For each $j\in\left\{  0,1,\ldots
,m-1\right\}  $, we have $e_{j}=\left(  \left[  p=j\right]  \right)  _{0\leq
p<m,\ 0\leq q<1}$ (by (\ref{eq.ej=})) and thus
\[
e_{j}^{T}=\left(  \left[  q=j\right]  \right)  _{0\leq p<1,\ 0\leq
q<m}\ \ \ \ \ \ \ \ \ \ \left(  \text{by the definition of the transpose of a
matrix}\right)  .
\]
Hence,%
\begin{align}
\sum_{j=0}^{m-1}A_{i,j}\underbrace{e_{j}^{T}}_{=\left(  \left[  q=j\right]
\right)  _{0\leq p<1,\ 0\leq q<m}}  &  =\sum_{j=0}^{m-1}A_{i,j}\left(  \left[
q=j\right]  \right)  _{0\leq p<1,\ 0\leq q<m}\nonumber\\
&  =\left(  \sum_{j=0}^{m-1}A_{i,j}\left[  q=j\right]  \right)  _{0\leq
p<1,\ 0\leq q<m}. \label{pf.lem.ei.row2.1}%
\end{align}
But for each $q\in\left\{  0,1,\ldots,m-1\right\}  $, we have%
\begin{align}
&  \underbrace{\sum_{j=0}^{m-1}}_{=\sum_{j\in\left\{  0,1,\ldots,m-1\right\}
}}A_{i,j}\left[  q=j\right] \nonumber\\
&  =\sum_{j\in\left\{  0,1,\ldots,m-1\right\}  }A_{i,j}\left[  q=j\right]
=A_{i,q}\underbrace{\left[  q=q\right]  }_{\substack{=1\\\text{(since
}q=q\text{)}}}+\sum_{\substack{j\in\left\{  0,1,\ldots,m-1\right\}  ;\\j\neq
q}}A_{i,j}\underbrace{\left[  q=j\right]  }_{\substack{=0\\\text{(since }q\neq
j\\\text{(because }j\neq q\text{))}}}\nonumber\\
&  \ \ \ \ \ \ \ \ \ \ \left(
\begin{array}
[c]{c}%
\text{here, we have split off the addend for }j=q\\
\text{from the sum (since }q\in\left\{  0,1,\ldots,m-1\right\}  \text{)}%
\end{array}
\right) \nonumber\\
&  =A_{i,q}+\underbrace{\sum_{\substack{j\in\left\{  0,1,\ldots,m-1\right\}
;\\j\neq q}}A_{i,j}0}_{=0}=A_{i,q}. \label{pf.lem.ei.row2.2}%
\end{align}
Hence, (\ref{pf.lem.ei.row2.1}) becomes
\begin{align*}
\sum_{j=0}^{m-1}A_{i,j}e_{j}^{T}  &  =\left(  \underbrace{\sum_{j=0}%
^{m-1}A_{i,j}\left[  q=j\right]  }_{\substack{=A_{i,q}\\\text{(by
(\ref{pf.lem.ei.row2.2}))}}}\right)  _{0\leq p<1,\ 0\leq q<m}=\left(
A_{i,q}\right)  _{0\leq p<1,\ 0\leq q<m}\\
&  =\left(
\begin{array}
[c]{cccc}%
A_{i,0} & A_{i,1} & \cdots & A_{i,m-1}%
\end{array}
\right) \\
&  =\left(  \text{the }i\text{-th row of }A\right)
\ \ \ \ \ \ \ \ \ \ \left(  \text{since }A=\left(  A_{i,j}\right)  _{0\leq
i<m,\ 0\leq j<m}\right) \\
&  =e_{i}^{T}A
\end{align*}
(since Lemma \ref{lem.ei.row} (applied to $\ell=m$) yields $e_{i}^{T}A=\left(
\text{the }i\text{-th row of }A\right)  $). This proves Lemma
\ref{lem.ei.row2}.
\end{proof}
\end{verlong}

\subsection{Proving $e_{u}^{T}SH_{c}=e_{u}^{T}H_{ac}$ for $u+1<m$}

We can now prove a generalization of Proposition \ref{prop.SHinf} to the case
of arbitrary $m$:

\begin{proposition}
\label{prop.SH}Let $u\in\mathbb{N}$ be such that $u+1<m$. Then:

\textbf{(a)} We have $e_{u}^{T}S=e_{u+1}^{T}$.

\textbf{(b)} Let $c\in\mathbb{L}$. Then, $e_{u}^{T}SH_{c}=e_{u}^{T}H_{ac}$.
\end{proposition}

\begin{vershort}
\begin{proof}
[Proof of Proposition \ref{prop.SH}.]\textbf{(a)} Lemma \ref{lem.ei.row}
(applied to $\ell=m$, $A=S$ and $i=u$) yields%
\begin{equation}
e_{u}^{T}S=\left(  \text{the }u\text{-th row of }S\right)  =\left(  \left[
q=u+1\right]  \right)  _{0\leq p<1,\ 0\leq q<m} \label{pf.prop.SH.short.1}%
\end{equation}
(by (\ref{eq.S=})). But (\ref{eq.ej=}) (applied to $j=u+1$) yields%
\[
e_{u+1}=\left(  \left[  p=u+1\right]  \right)  _{0\leq p<m,\ 0\leq q<1}.
\]
Thus, by the definition of the transpose of a matrix, we obtain%
\[
e_{u+1}^{T}=\left(  \left[  q=u+1\right]  \right)  _{0\leq p<1,\ 0\leq q<m}.
\]
Comparing this with (\ref{pf.prop.SH.short.1}), we obtain $e_{u}^{T}%
S=e_{u+1}^{T}$. This proves Proposition \ref{prop.SH} \textbf{(a)}.

\textbf{(b)} Lemma \ref{lem.ei.row} (applied to $\ell=1$, $A=H_{ac}$ and
$i=u$) yields%
\begin{align*}
e_{u}^{T}H_{ac}  &  =\left(  \text{the }u\text{-th row of }H_{ac}\right) \\
&  =\left(  \text{the }u\text{-th entry of }H_{ac}\right)
\ \ \ \ \ \ \ \ \ \ \left(  \text{since }H_{ac}\text{ is a column
vector}\right) \\
&  =\underbrace{a^{u}a}_{=a^{u+1}}c\ \ \ \ \ \ \ \ \ \ \left(  \text{since
(\ref{eq.Hc=}) yields }H_{ac}=\left(  a^{i}ac\right)  _{0\leq i<m,\ 0\leq
j<1}\right) \\
&  =a^{u+1}c.
\end{align*}
Comparing this with%
\begin{align*}
&  \underbrace{e_{u}^{T}S}_{\substack{=e_{u+1}^{T}\\\text{(by Proposition
\ref{prop.SH} \textbf{(a)})}}}H_{c}\\
&  =e_{u+1}^{T}H_{c}=\left(  \text{the }\left(  u+1\right)  \text{-th row of
}H_{c}\right) \\
&  \ \ \ \ \ \ \ \ \ \ \left(  \text{by Lemma \ref{lem.ei.row}, applied to
}\ell=1\text{, }A=H_{c}\text{ and }i=u+1\right) \\
&  =\left(  \text{the }\left(  u+1\right)  \text{-th entry of }H_{c}\right)
\ \ \ \ \ \ \ \ \ \ \left(  \text{since }H_{c}\text{ is a column
vector}\right) \\
&  =a^{u+1}c\ \ \ \ \ \ \ \ \ \ \left(  \text{since (\ref{eq.Hc=}) yields
}H_{c}=\left(  a^{i}c\right)  _{0\leq i<m,\ 0\leq j<1}\right)  ,
\end{align*}
we obtain $e_{u}^{T}SH_{c}=e_{u}^{T}H_{ac}$. This proves Proposition
\ref{prop.SH} \textbf{(b)}.
\end{proof}
\end{vershort}

\begin{verlong}
\begin{proof}
[Proof of Proposition \ref{prop.SH}.]We have $u\in\mathbb{N}$, hence
$u+1\in\mathbb{N}$, thus $u+1\geq0$, hence $0\leq u+1<m$.

Also, $u\geq0$ (since $u\in\mathbb{N}$), thus $0\leq u<m$ (since $u<u+1<m$).

\textbf{(a)} Recall that $0\leq u+1<m$. Therefore, (\ref{eq.ej=}) (applied to
$j=u+1$) yields%
\[
e_{u+1}=\left(  \left[  p=u+1\right]  \right)  _{0\leq p<m,\ 0\leq q<1}.
\]
Thus, by the definition of the transpose of a matrix, we obtain%
\begin{equation}
e_{u+1}^{T}=\left(  \left[  q=u+1\right]  \right)  _{0\leq p<1,\ 0\leq
q<m}=\left(  \left[  j=u+1\right]  \right)  _{0\leq p<1,\ 0\leq j<m}
\label{pf.prop.SH.1}%
\end{equation}
(here, we have renamed the index $\left(  p,q\right)  $ as $\left(
p,j\right)  $).

Recall that $0\leq u<m$. Thus, Lemma \ref{lem.ei.row} (applied to $\ell=m$,
$A=S$ and $i=u$) yields%
\[
e_{u}^{T}S=\left(  \text{the }u\text{-th row of }S\right)  =\left(  \left[
j=u+1\right]  \right)  _{0\leq p<1,\ 0\leq j<m}\ \ \ \ \ \ \ \ \ \ \left(
\text{by (\ref{eq.S=})}\right)
\]
Comparing this with (\ref{pf.prop.SH.1}), we obtain $e_{u}^{T}S=e_{u+1}^{T}$.
This proves Proposition \ref{prop.SH} \textbf{(a)}.

\textbf{(b)} From (\ref{eq.Hc=}) (applied to $ac$ instead of $c$), we obtain%
\[
H_{ac}=\left(  a^{i}ac\right)  _{0\leq i<m,\ 0\leq j<1}.
\]
Hence,%
\begin{align*}
\left(  \text{the }u\text{-th entry of }H_{ac}\right)   &  =\underbrace{a^{u}%
a}_{=a^{u+1}}c\ \ \ \ \ \ \ \ \ \ \left(  \text{since }0\leq u<m\right) \\
&  =a^{u+1}c.
\end{align*}

But Lemma \ref{lem.ei.row} (applied to $\ell=1$, $A=H_{ac}$ and $i=u$) yields%
\begin{align*}
e_{u}^{T}H_{ac}  &  =\left(  \text{the }u\text{-th row of }H_{ac}\right) \\
&  =\left(  \text{the }u\text{-th entry of }H_{ac}\right)
\ \ \ \ \ \ \ \ \ \ \left(  \text{since }H_{ac}\text{ is a column
vector}\right) \\
&  =a^{u+1}c.
\end{align*}
Comparing this with%
\begin{align*}
&  \underbrace{e_{u}^{T}S}_{\substack{=e_{u+1}^{T}\\\text{(by Proposition
\ref{prop.SH} \textbf{(a)})}}}H_{c}\\
&  =e_{u+1}^{T}H_{c}=\left(  \text{the }\left(  u+1\right)  \text{-th row of
}H_{c}\right) \\
&  \ \ \ \ \ \ \ \ \ \ \left(  \text{by Lemma \ref{lem.ei.row}, applied to
}\ell=1\text{, }A=H_{c}\text{ and }i=u+1\right) \\
&  =\left(  \text{the }\left(  u+1\right)  \text{-th entry of }H_{c}\right)
\ \ \ \ \ \ \ \ \ \ \left(  \text{since }H_{c}\text{ is a column
vector}\right) \\
&  =a^{u+1}c\ \ \ \ \ \ \ \ \ \ \left(  \text{by (\ref{eq.Hc=})}\right)  ,
\end{align*}
we obtain $e_{u}^{T}SH_{c}=e_{u}^{T}H_{ac}$. This proves Proposition
\ref{prop.SH} \textbf{(b)}.
\end{proof}
\end{verlong}

It is now easy to derive Proposition \ref{prop.SHinf} from Proposition
\ref{prop.SH} \textbf{(b)}:

\begin{proof}
[Proof of Proposition \ref{prop.SHinf} (sketched).]We have $m=\infty$; thus,
every $u\in\mathbb{N}$ satisfies $u+1<m$. Hence, Proposition \ref{prop.SH}
\textbf{(b)} yields that $e_{u}^{T}SH_{c}=e_{u}^{T}H_{ac}$ for every
$u\in\mathbb{N}$. From this, it is easy to conclude that $SH_{c}=H_{ac}$
(using Lemma \ref{lem.ei.row}). We leave the details to the reader, since we
will not use Proposition \ref{prop.SHinf}.
\end{proof}

\subsection{Proving $U_{b}H_{c}=H_{bc}$}

Next, we shall prove Proposition \ref{prop.UHinf}. For convenience, let us
recall its statement:

\begin{proposition}
\label{prop.UH}Let $b\in\mathbb{L}$ and $c\in\mathbb{L}$. Then, $U_{b}%
H_{c}=H_{bc}$.
\end{proposition}

\begin{proof}
[Proof of Proposition \ref{prop.UH}.]Let $u\in\left\{  0,1,\ldots,m-1\right\}
$. Hence, $0\leq u\leq m-1$. (Keep in mind that $\left\{  0,1,\ldots
,\infty-1\right\}  =\mathbb{N}$, so $u$ cannot be $\infty$ even when
$m=\infty$.)

From (\ref{eq.Ub=}), we see that%
\begin{equation}
\left(  U_{b}\right)  _{i,j}=%
\begin{cases}
\dbinom{i}{j}\operatorname*{ad}\nolimits_{a}^{i-j}\left(  b\right)  , &
\text{if }i\geq j;\\
0, & \text{if }i<j
\end{cases}
\label{pf.prop.UH.Ubij=}%
\end{equation}
for each $i\in\left\{  0,1,\ldots,m-1\right\}  $ and $j\in\left\{
0,1,\ldots,m-1\right\}  $.

From (\ref{eq.Hc=}), we obtain%
\begin{equation}
\left(  H_{c}\right)  _{i,0}=a^{i}c \label{pf.prop.UH.Hci=}%
\end{equation}
for each $i\in\left\{  0,1,\ldots,m-1\right\}  $. The same argument (applied
to $bc$ instead of $c$) yields%
\begin{equation}
\left(  H_{bc}\right)  _{i,0}=a^{i}bc \label{pf.prop.UH.Hbci=}%
\end{equation}
for each $i\in\left\{  0,1,\ldots,m-1\right\}  $.

Now, (\ref{eq.ABik=}) (applied to $m$, $m$, $1$, $U_{b}$, $H_{c}$, $u$ and $0$
instead of $u$, $v$, $w$, $A$, $B$, $i$ and $k$) yields%
\begin{align*}
\left(  U_{b}H_{c}\right)  _{u,0}  &  =\sum_{j=0}^{m-1}\underbrace{\left(
U_{b}\right)  _{u,j}}_{\substack{=%
\begin{cases}
\dbinom{u}{j}\operatorname*{ad}\nolimits_{a}^{u-j}\left(  b\right)  , &
\text{if }u\geq j;\\
0, & \text{if }u<j
\end{cases}
\\\text{(by (\ref{pf.prop.UH.Ubij=}), applied to }i=u\text{)}}%
}\ \ \ \underbrace{\left(  H_{c}\right)  _{j,0}}_{\substack{=a^{j}c\\\text{(by
(\ref{pf.prop.UH.Hci=}), applied to }i=j\text{)}}}\\
&  =\sum_{j=0}^{m-1}%
\begin{cases}
\dbinom{u}{j}\operatorname*{ad}\nolimits_{a}^{u-j}\left(  b\right)  , &
\text{if }u\geq j;\\
0, & \text{if }u<j
\end{cases}
\cdot a^{j}c\\
&  =\sum_{j=0}^{u}\underbrace{%
\begin{cases}
\dbinom{u}{j}\operatorname*{ad}\nolimits_{a}^{u-j}\left(  b\right)  , &
\text{if }u\geq j;\\
0, & \text{if }u<j
\end{cases}
}_{\substack{=\dbinom{u}{j}\operatorname*{ad}\nolimits_{a}^{u-j}\left(
b\right)  \\\text{(since }u\geq j\text{ (because }j\leq u\text{))}}}\cdot
a^{j}c+\sum_{j=u+1}^{m-1}\underbrace{%
\begin{cases}
\dbinom{u}{j}\operatorname*{ad}\nolimits_{a}^{u-j}\left(  b\right)  , &
\text{if }u\geq j;\\
0, & \text{if }u<j
\end{cases}
}_{\substack{=0\\\text{(since }u<j\text{ (because }j\geq u+1>u\text{))}}}\cdot
a^{j}c\\
&  \ \ \ \ \ \ \ \ \ \ \left(  \text{here, we have split the sum at
}j=u\text{, since }0\leq u\leq m-1\right) \\
&  =\sum_{j=0}^{u}\dbinom{u}{j}\operatorname*{ad}\nolimits_{a}^{u-j}\left(
b\right)  \cdot a^{j}c+\underbrace{\sum_{j=u+1}^{m-1}0\cdot a^{j}c}_{=0}%
=\sum_{j=0}^{u}\dbinom{u}{j}\operatorname*{ad}\nolimits_{a}^{u-j}\left(
b\right)  \cdot a^{j}c.
\end{align*}
Comparing this with%
\begin{align*}
\left(  H_{bc}\right)  _{u,0}  &  =\underbrace{a^{u}b}_{\substack{=\sum
_{j=0}^{u}\dbinom{u}{j}\operatorname*{ad}\nolimits_{a}^{u-j}\left(  b\right)
\cdot a^{j}\\\text{(by Lemma \ref{lem.ada1}, applied to }i=u\text{)}%
}}c\ \ \ \ \ \ \ \ \ \ \left(  \text{by (\ref{pf.prop.UH.Hbci=}), applied to
}i=u\right) \\
&  =\left(  \sum_{j=0}^{u}\dbinom{u}{j}\operatorname*{ad}\nolimits_{a}%
^{u-j}\left(  b\right)  \cdot a^{j}\right)  c=\sum_{j=0}^{u}\dbinom{u}%
{j}\operatorname*{ad}\nolimits_{a}^{u-j}\left(  b\right)  \cdot a^{j}c,
\end{align*}
we obtain $\left(  U_{b}H_{c}\right)  _{u,0}=\left(  H_{bc}\right)  _{u,0}$.

Now, recall that $U_{b}H_{c}$ is a column vector. Hence,%
\begin{equation}
\left(  \text{the }u\text{-th entry of }U_{b}H_{c}\right)  =\left(  U_{b}%
H_{c}\right)  _{u,0}=\left(  H_{bc}\right)  _{u,0}. \label{pf.prop.UH.7}%
\end{equation}
But $H_{bc}$ is also a column vector. Thus,%
\[
\left(  \text{the }u\text{-th entry of }H_{bc}\right)  =\left(  H_{bc}\right)
_{u,0}.
\]
Comparing this with (\ref{pf.prop.UH.7}), we obtain%
\[
\left(  \text{the }u\text{-th entry of }U_{b}H_{c}\right)  =\left(  \text{the
}u\text{-th entry of }H_{bc}\right)  .
\]

Now, forget that we fixed $u$. We thus have shown that $\left(  \text{the
}u\text{-th entry of }U_{b}H_{c}\right)  =\left(  \text{the }u\text{-th entry
of }H_{bc}\right)  $ for each $u\in\left\{  0,1,\ldots,m-1\right\}  $. In
other words, each entry of $U_{b}H_{c}$ equals the corresponding entry of
$H_{bc}$. Thus, the two column vectors $U_{b}H_{c}$ and $H_{bc}$ are
identical. In other words, $U_{b}H_{c}=H_{bc}$. This proves Proposition
\ref{prop.UH}.
\end{proof}

\subsection{The $\protect\underset{k}{\equiv}$ relations}

Now, we introduce a notation for saying that two $m\times\ell$-matrices are
equal in their first $m-k+1$ rows:

\begin{definition}
\label{def.equivk}Let $\ell\in\mathbb{N}\cup\left\{  \infty\right\}  $. Let
$A\in\mathbb{L}^{m\times\ell}$ and $B\in\mathbb{L}^{m\times\ell}$ be two
$m\times\ell$-matrices. Let $k$ be a positive integer. We shall say that
$A\underset{k}{\equiv}B$ if and only if we have%
\[
\left(  e_{u}^{T}A=e_{u}^{T}B\ \ \ \ \ \ \ \ \ \ \text{for all }u\in\left\{
0,1,\ldots,m-k\right\}  \right)  .
\]

\end{definition}

(Recall again that $\left\{  0,1,\ldots,\infty\right\}  $ means $\mathbb{N}$;
thus, \textquotedblleft$u\in\left\{  0,1,\ldots,m-k\right\}  $%
\textquotedblright\ means \textquotedblleft$u\in\mathbb{N}$\textquotedblright%
\ in the case when $m=\infty$. Note that $\left\{  0,1,\ldots,g\right\}  $
means the empty set $\varnothing$ when $g<0$.)

Note that the condition \textquotedblleft$e_{u}^{T}A=e_{u}^{T}B$%
\textquotedblright\ in Definition \ref{def.equivk} can be restated as
\textquotedblleft the $u$-th row of $A$ equals the $u$-th row of
$B$\textquotedblright, because of Lemma \ref{lem.ei.row}. But we will find it
easier to use it in the form \textquotedblleft$e_{u}^{T}A=e_{u}^{T}%
B$\textquotedblright.

The following lemma is easy:

\begin{lemma}
\label{lem.equivk.Ub}Let $\ell\in\mathbb{N}\cup\left\{  \infty\right\}  $. Let
$A\in\mathbb{L}^{m\times\ell}$ and $B\in\mathbb{L}^{m\times\ell}$ be two
$m\times\ell$-matrices. Let $k$ be a positive integer such that
$A\underset{k}{\equiv}B$. Let $b\in\mathbb{L}$. Then, $U_{b}%
A\underset{k}{\equiv}U_{b}B$.
\end{lemma}

All that is needed of the matrix $U_{b}$ for Lemma \ref{lem.equivk.Ub} to hold
is that $U_{b}$ is lower-triangular; we stated it for $U_{b}$ just for
convenience reasons.

\begin{proof}
[Proof of Lemma \ref{lem.equivk.Ub}.]We have $A\underset{k}{\equiv}B$. In
other words, we have%
\begin{equation}
\left(  e_{u}^{T}A=e_{u}^{T}B\ \ \ \ \ \ \ \ \ \ \text{for all }u\in\left\{
0,1,\ldots,m-k\right\}  \right)  \label{pf.lem.equivk.Ub.ass}%
\end{equation}
(by the definition of \textquotedblleft$A\underset{k}{\equiv}B$%
\textquotedblright).

\begin{verlong}
We have $k\geq1$ (since $k$ is a positive integer).
\end{verlong}

Let $u\in\left\{  0,1,\ldots,m-k\right\}  $. Thus, $0\leq u\leq m-k$. Also,
$u\in\left\{  0,1,\ldots,m-k\right\}  \subseteq\left\{  0,1,\ldots
,m-1\right\}  $ (since $m-\underbrace{k}_{\geq1}\leq m-1$). Hence, $0\leq
u\leq m-1<m$. For each $j\in\left\{  0,1,\ldots,u\right\}  $, we have
$j\in\left\{  0,1,\ldots,u\right\}  \subseteq\left\{  0,1,\ldots,m-k\right\}
$ (since $u\leq m-k$) and thus%
\begin{equation}
e_{j}^{T}A=e_{j}^{T}B \label{pf.lem.equivk.Ub.assv}%
\end{equation}
(by (\ref{pf.lem.equivk.Ub.ass}), applied to $j$ instead of $u$).

From (\ref{eq.Ub=}), we see that%
\begin{equation}
\left(  U_{b}\right)  _{i,j}=%
\begin{cases}
\dbinom{i}{j}\operatorname*{ad}\nolimits_{a}^{i-j}\left(  b\right)  , &
\text{if }i\geq j;\\
0, & \text{if }i<j
\end{cases}
\label{pf.lem.equivk.Ub.Ubij=}%
\end{equation}
for each $i\in\left\{  0,1,\ldots,m-1\right\}  $ and $j\in\left\{
0,1,\ldots,m-1\right\}  $.

For each $j\in\left\{  u+1,u+2,\ldots,m-1\right\}  $, we have $j\geq u+1$ and%
\begin{align}
\left(  U_{b}\right)  _{u,j}  &  =%
\begin{cases}
\dbinom{u}{j}\operatorname*{ad}\nolimits_{a}^{u-j}\left(  b\right)  , &
\text{if }u\geq j;\\
0, & \text{if }u<j
\end{cases}
\ \ \ \ \ \ \ \ \ \ \left(  \text{by (\ref{pf.lem.equivk.Ub.Ubij=}), applied
to }i=u\right) \nonumber\\
&  =0\ \ \ \ \ \ \ \ \ \ \left(  \text{since }u<j\text{ (because }j\geq
u+1>u\text{)}\right)  . \label{pf.lem.equivk.Ub.lav=0}%
\end{align}

Now, Lemma \ref{lem.ei.row2} (applied to $i=u$ and $A=U_{b}$) yields%
\begin{align*}
e_{u}^{T}U_{b}  &  =\sum_{j=0}^{m-1}\left(  U_{b}\right)  _{u,j}e_{j}^{T}%
=\sum_{j=0}^{u}\left(  U_{b}\right)  _{u,j}e_{j}^{T}+\sum_{j=u+1}%
^{m-1}\underbrace{\left(  U_{b}\right)  _{u,j}}_{\substack{=0\\\text{(by
(\ref{pf.lem.equivk.Ub.lav=0}))}}}e_{j}^{T}\\
&  \ \ \ \ \ \ \ \ \ \ \left(  \text{here, we have split the sum at
}j=u\text{, since }0\leq u\leq m-1\right) \\
&  =\sum_{j=0}^{u}\left(  U_{b}\right)  _{u,j}e_{j}^{T}+\underbrace{\sum
_{j=u+1}^{m-1}0e_{j}^{T}}_{=0}=\sum_{j=0}^{u}\left(  U_{b}\right)  _{u,j}%
e_{j}^{T}.
\end{align*}
Hence,%
\[
\underbrace{e_{u}^{T}U_{b}}_{=\sum_{j=0}^{u}\left(  U_{b}\right)  _{u,j}%
e_{j}^{T}}A=\left(  \sum_{j=0}^{u}\left(  U_{b}\right)  _{u,j}e_{j}%
^{T}\right)  A=\sum_{j=0}^{u}\left(  U_{b}\right)  _{u,j}\underbrace{e_{j}%
^{T}A}_{\substack{=e_{j}^{T}B\\\text{(by (\ref{pf.lem.equivk.Ub.assv}))}%
}}=\sum_{j=0}^{u}\left(  U_{b}\right)  _{u,j}e_{j}^{T}B.
\]
Comparing this with%
\[
\underbrace{e_{u}^{T}U_{b}}_{=\sum_{j=0}^{u}\left(  U_{b}\right)  _{u,j}%
e_{j}^{T}}B=\left(  \sum_{j=0}^{u}\left(  U_{b}\right)  _{u,j}e_{j}%
^{T}\right)  B=\sum_{j=0}^{u}\left(  U_{b}\right)  _{u,j}e_{j}^{T}B,
\]
we obtain $e_{u}^{T}U_{b}A=e_{u}^{T}U_{b}B$.

Forget that we fixed $u$. We thus have shown that%
\[
\left(  e_{u}^{T}U_{b}A=e_{u}^{T}U_{b}B\ \ \ \ \ \ \ \ \ \ \text{for all }%
u\in\left\{  0,1,\ldots,m-k\right\}  \right)  .
\]
In other words, $U_{b}A\underset{k}{\equiv}U_{b}B$ (by the definition of
\textquotedblleft$U_{b}A\underset{k}{\equiv}U_{b}B$\textquotedblright). This
proves Lemma \ref{lem.equivk.Ub}.
\end{proof}

The analogue of Lemma \ref{lem.equivk.Ub} for $S$ is even simpler:

\begin{lemma}
\label{lem.equivk.S}Let $\ell\in\mathbb{N}\cup\left\{  \infty\right\}  $. Let
$A\in\mathbb{L}^{m\times\ell}$ and $B\in\mathbb{L}^{m\times\ell}$ be two
$m\times\ell$-matrices. Let $k$ be a positive integer such that
$A\underset{k}{\equiv}B$. Then, $SA\underset{k+1}{\equiv}SB$.
\end{lemma}

\begin{proof}
[Proof of Lemma \ref{lem.equivk.S}.]We have $A\underset{k}{\equiv}B$. In other
words, we have%
\begin{equation}
\left(  e_{u}^{T}A=e_{u}^{T}B\ \ \ \ \ \ \ \ \ \ \text{for all }u\in\left\{
0,1,\ldots,m-k\right\}  \right)  \label{pf.lem.equivk.S.ass}%
\end{equation}
(by the definition of \textquotedblleft$A\underset{k}{\equiv}B$%
\textquotedblright).

Let $u\in\left\{  0,1,\ldots,m-\left(  k+1\right)  \right\}  $. Then, $u\leq
m-\left(  k+1\right)  =m-k-1$, so that $u+1\leq m-k$. Combining this with
$u+1\in\mathbb{N}$ (since $u\in\left\{  0,1,\ldots,m-\left(  k+1\right)
\right\}  \subseteq\mathbb{N}$), we obtain $u+1\in\left\{  0,1,\ldots
,m-k\right\}  $. Hence, (\ref{pf.lem.equivk.S.ass}) (applied to $u+1$ instead
of $u$) yields $e_{u+1}^{T}A=e_{u+1}^{T}B$.

But $u+1\leq m-\underbrace{k}_{>0}<m$. Hence, Proposition \ref{prop.SH}
\textbf{(a)} yields $e_{u}^{T}S=e_{u+1}^{T}$. Hence, $\underbrace{e_{u}^{T}%
S}_{=e_{u+1}^{T}}A=e_{u+1}^{T}A=e_{u+1}^{T}B$. Comparing this with
$\underbrace{e_{u}^{T}S}_{=e_{u+1}^{T}}B=e_{u+1}^{T}B$, we obtain $e_{u}%
^{T}SA=e_{u}^{T}SB$.

Forget that we fixed $u$. We thus have shown that%
\[
\left(  e_{u}^{T}SA=e_{u}^{T}SB\ \ \ \ \ \ \ \ \ \ \text{for all }u\in\left\{
0,1,\ldots,m-\left(  k+1\right)  \right\}  \right)  .
\]
In other words, $SA\underset{k+1}{\equiv}SB$ (by the definition of
\textquotedblleft$SA\underset{k+1}{\equiv}SB$\textquotedblright). This proves
Lemma \ref{lem.equivk.S}.
\end{proof}

Now, we can prove the following lemma, which is as close as we can get to
Proposition \ref{prop.SHinf} without requiring $m=\infty$:

\begin{lemma}
\label{lem.equivk.Sc}Let $A\in\mathbb{L}^{m\times1}$ and $c\in\mathbb{L}$. Let
$k$ be a positive integer such that $A\underset{k}{\equiv}H_{c}$. Then,
$SA\underset{k+1}{\equiv}H_{ac}$.
\end{lemma}

\begin{proof}
[Proof of Lemma \ref{lem.equivk.Sc}.]Lemma \ref{lem.equivk.S} (applied to
$\ell=1$ and $B=H_{c}$) yields $SA\underset{k+1}{\equiv}SH_{c}$. In other
words, we have
\begin{equation}
\left(  e_{u}^{T}SA=e_{u}^{T}SH_{c}\ \ \ \ \ \ \ \ \ \ \text{for all }%
u\in\left\{  0,1,\ldots,m-\left(  k+1\right)  \right\}  \right)
\label{pf.lem.equivk.Sc.ass}%
\end{equation}
(by the definition of \textquotedblleft$SA\underset{k+1}{\equiv}SH_{c}%
$\textquotedblright).

Now, let $u\in\left\{  0,1,\ldots,m-\left(  k+1\right)  \right\}  $. Thus,
$u\leq m-\left(  k+1\right)  =m-k-1$, so that $u+1\leq m-\underbrace{k}%
_{>0}<m$. Thus, Proposition \ref{prop.SH} \textbf{(b)} yields $e_{u}^{T}%
SH_{c}=e_{u}^{T}H_{ac}$. But (\ref{pf.lem.equivk.Sc.ass}) yields%
\[
e_{u}^{T}SA=e_{u}^{T}SH_{c}=e_{u}^{T}H_{ac}.
\]

Now, forget that we fixed $u$. We thus have shown that
\[
\left(  e_{u}^{T}SA=e_{u}^{T}H_{ac}\ \ \ \ \ \ \ \ \ \ \text{for all }%
u\in\left\{  0,1,\ldots,m-\left(  k+1\right)  \right\}  \right)  .
\]
In other words, $SA\underset{k+1}{\equiv}H_{ac}$ (by the definition of
\textquotedblleft$SA\underset{k+1}{\equiv}H_{ac}$\textquotedblright). This
proves Lemma \ref{lem.equivk.Sc}.
\end{proof}

\subsection{Proof of Theorem \ref{thm.gen}}

Our last stop before Theorem \ref{thm.gen} is the following lemma, which by
now is an easy induction:

\begin{lemma}
\label{lem.gen-lem}Let $b\in\mathbb{L}$. Let $n\in\mathbb{N}$. Then,%
\[
\left(  U_{b}S\right)  ^{n}H_{1}\underset{n+1}{\equiv}H_{\left(  ba\right)
^{n}}.
\]

\end{lemma}

\begin{proof}
[Proof of Lemma \ref{lem.gen-lem}.]We shall prove Lemma \ref{lem.gen-lem} by
induction on $n$:

\textit{Induction base:} It is easy to see that $H_{1}\underset{0+1}{\equiv
}H_{1}$\ \ \ \ \footnote{\textit{Proof.} Clearly,%
\[
\left(  e_{u}^{T}H_{1}=e_{u}^{T}H_{1}\ \ \ \ \ \ \ \ \ \ \text{for all }%
u\in\left\{  0,1,\ldots,m-\left(  0+1\right)  \right\}  \right)  .
\]
In other words, $H_{1}\underset{0+1}{\equiv}H_{1}$ (by the definition of
\textquotedblleft$H_{1}\underset{0+1}{\equiv}H_{1}$\textquotedblright).}.

But $\underbrace{\left(  U_{b}S\right)  ^{0}}_{=I_{m}}H_{1}=I_{m}H_{1}=H_{1}$
and $H_{\left(  ba\right)  ^{0}}=H_{1}$ (since $\left(  ba\right)  ^{0}=1$).
In view of these two equalities, we can rewrite $H_{1}\underset{0+1}{\equiv
}H_{1}$ as $\left(  U_{b}S\right)  ^{0}H_{1}\underset{0+1}{\equiv}H_{\left(
ba\right)  ^{0}}$. In other words, Lemma \ref{lem.gen-lem} holds for $n=0$.
This completes the induction base.

\textit{Induction step:} Let $k$ be a positive integer. Assume that Lemma
\ref{lem.gen-lem} holds for $n=k-1$. We must prove that Lemma
\ref{lem.gen-lem} holds for $n=k$.

\begin{vershort}
We have assumed that Lemma \ref{lem.gen-lem} holds for $n=k-1$. In other
words, we have%
\[
\left(  U_{b}S\right)  ^{k-1}H_{1}\underset{k}{\equiv}H_{\left(  ba\right)
^{k-1}}.
\]

\end{vershort}

\begin{verlong}
We have assumed that Lemma \ref{lem.gen-lem} holds for $n=k-1$. In other
words, we have%
\[
\left(  U_{b}S\right)  ^{k-1}H_{1}\underset{\left(  k-1\right)  +1}{\equiv
}H_{\left(  ba\right)  ^{k-1}}.
\]
In other words, we have%
\[
\left(  U_{b}S\right)  ^{k-1}H_{1}\underset{k}{\equiv}H_{\left(  ba\right)
^{k-1}}%
\]
(since $\left(  k-1\right)  +1=k$).
\end{verlong}

Hence, Lemma \ref{lem.equivk.Sc} (applied to $A=\left(  U_{b}S\right)
^{k-1}H_{1}$ and $c=\left(  ba\right)  ^{k-1}$) yields
\[
S\left(  U_{b}S\right)  ^{k-1}H_{1}\underset{k+1}{\equiv}H_{a\left(
ba\right)  ^{k-1}}.
\]
Thus, Lemma \ref{lem.equivk.Ub} (applied to $1$, $k+1$, $S\left(
U_{b}S\right)  ^{k-1}H_{1}$ and $H_{a\left(  ba\right)  ^{k-1}}$ instead of
$\ell$, $k$, $A$ and $B$) yields%
\[
U_{b}S\left(  U_{b}S\right)  ^{k-1}H_{1}\underset{k+1}{\equiv}U_{b}H_{a\left(
ba\right)  ^{k-1}}.
\]
In view of%
\[
U_{b}S\left(  U_{b}S\right)  ^{k-1}=\left(  U_{b}S\right)  \left(
U_{b}S\right)  ^{k-1}=\left(  U_{b}S\right)  ^{k}%
\]
and%
\begin{align*}
U_{b}H_{a\left(  ba\right)  ^{k-1}}  &  =H_{ba\left(  ba\right)  ^{k-1}%
}\ \ \ \ \ \ \ \ \ \ \left(  \text{by Proposition \ref{prop.UH}, applied to
}c=a\left(  ba\right)  ^{k-1}\right) \\
&  =H_{\left(  ba\right)  ^{k}}\ \ \ \ \ \ \ \ \ \ \left(  \text{since
}ba\left(  ba\right)  ^{k-1}=\left(  ba\right)  \left(  ba\right)
^{k-1}=\left(  ba\right)  ^{k}\right)  ,
\end{align*}
this rewrites as
\[
\left(  U_{b}S\right)  ^{k}H_{1}\underset{k+1}{\equiv}H_{\left(  ba\right)
^{k}}.
\]
In other words, Lemma \ref{lem.gen-lem} holds for $n=k$. This completes the
induction step. Thus, Lemma \ref{lem.gen-lem} is proven by induction.
\end{proof}

We can now easily prove Theorem \ref{thm.gen}:

\begin{proof}
[Proof of Theorem \ref{thm.gen}.]We have $n<m$, thus $n+1\leq m$ (since $n$
and $m$ are integers), hence $m-\left(  n+1\right)  \geq0$. Thus,
$0\in\left\{  0,1,\ldots,m-\left(  n+1\right)  \right\}  $. Also, $0\leq0<m$
(since $0\leq n<m$).

Lemma \ref{lem.gen-lem} shows that
\[
\left(  U_{b}S\right)  ^{n}H_{1}\underset{n+1}{\equiv}H_{\left(  ba\right)
^{n}}.
\]
In other words, we have
\[
\left(  e_{u}^{T}\left(  U_{b}S\right)  ^{n}H_{1}=e_{u}^{T}H_{\left(
ba\right)  ^{n}}\ \ \ \ \ \ \ \ \ \ \text{for all }u\in\left\{  0,1,\ldots
,m-\left(  n+1\right)  \right\}  \right)
\]
(by the definition of \textquotedblleft$\left(  U_{b}S\right)  ^{n}%
H_{1}\underset{n+1}{\equiv}H_{\left(  ba\right)  ^{n}}$\textquotedblright). We
can apply this to $u=0$ (since $0\in\left\{  0,1,\ldots,m-\left(  n+1\right)
\right\}  $), and thus obtain%
\begin{align*}
e_{0}^{T}\left(  U_{b}S\right)  ^{n}H_{1}  &  =e_{0}^{T}H_{\left(  ba\right)
^{n}}=\left(  \text{the }0\text{-th row of }H_{\left(  ba\right)  ^{n}}\right)
\\
&  \ \ \ \ \ \ \ \ \ \ \left(  \text{by Lemma \ref{lem.ei.row}, applied to
}\ell=1\text{, }i=0\text{ and }A=H_{\left(  ba\right)  ^{n}}\right) \\
&  =\left(  \text{the }0\text{-th entry of }H_{\left(  ba\right)  ^{n}%
}\right)  \ \ \ \ \ \ \ \ \ \ \left(  \text{since }H_{\left(  ba\right)  ^{n}%
}\text{ is a column vector}\right) \\
&  =\underbrace{a^{0}}_{=1}\left(  ba\right)  ^{n}\ \ \ \ \ \ \ \ \ \ \left(
\begin{array}
[c]{c}%
\text{since (\ref{eq.Hc=}) (applied to }c=\left(  ba\right)  ^{n}\text{)}\\
\text{yields }H_{\left(  ba\right)  ^{n}}=\left(  a^{i}\left(  ba\right)
^{n}\right)  _{0\leq i<m,\ 0\leq j<1}%
\end{array}
\right) \\
&  =\left(  ba\right)  ^{n}.
\end{align*}
This proves Theorem \ref{thm.gen}.
\end{proof}

\section{\label{sect.weyl}A Weyl-algebraic application}

\subsection{The claim}

We shall now restrict ourselves to a more special situation.

Namely, we let $\mathbb{K}$ be a commutative ring, and we assume that the ring
$\mathbb{L}$ is a $\mathbb{K}$-algebra.

Consider the polynomial ring $\mathbb{K}\left[  t\right]  $ in one variable
$t$ over $\mathbb{K}$. For each polynomial $g\in\mathbb{K}\left[  t\right]  $
and each $n\in\mathbb{N}$, we let $g^{\left(  n\right)  }$ be the $n$-th
derivative of $g$; that is,
\begin{equation}
g^{\left(  n\right)  }=\dfrac{d^{n}}{dt^{n}}g. \label{eq.gn=}%
\end{equation}
Thus, in particular, $g^{\left(  0\right)  }=g$ and $g^{\left(  1\right)
}=g^{\prime}$ (where $g^{\prime}$ denotes the derivative $\dfrac{d}{dt}g$ of
$g$).

Recall that we fixed $a\in\mathbb{L}$. Furthermore, let $h\in\mathbb{L}$ and
$x\in\mathbb{L}$ be such that
\[
\left[  a,x\right]  =h\ \ \ \ \ \ \ \ \ \ \text{and}%
\ \ \ \ \ \ \ \ \ \ \left[  h,a\right]  =0\ \ \ \ \ \ \ \ \ \ \text{and}%
\ \ \ \ \ \ \ \ \ \ \left[  h,x\right]  =0.
\]
This situation is actually fairly common:

\begin{example}
\label{exa.weyl.weyl}Let $D$ be the differentiation operator
\[
\mathbb{K}\left[  t\right]  \rightarrow\mathbb{K}\left[  t\right]
,\ \ \ \ \ \ \ \ \ \ g\mapsto\dfrac{d}{dt}g.
\]
Let $T$ be the \textquotedblleft multiplication by $t$\textquotedblright%
\ operator%
\[
\mathbb{K}\left[  t\right]  \rightarrow\mathbb{K}\left[  t\right]
,\ \ \ \ \ \ \ \ \ \ g\mapsto tg.
\]
Then, the three operators $D$, $T$ and $\operatorname*{id}%
\nolimits_{\mathbb{K}\left[  t\right]  }$ belong to the $\mathbb{K}$-algebra
$\operatorname*{End}\nolimits_{\mathbb{K}}\left(  \mathbb{K}\left[  t\right]
\right)  $ of all endomorphisms of the $\mathbb{K}$-module $\mathbb{K}\left[
t\right]  $. These three operators satisfy%
\[
\left[  D,T\right]  =\operatorname*{id}\nolimits_{\mathbb{K}\left[  t\right]
},\ \ \ \ \ \ \ \ \ \ \left[  \operatorname*{id}\nolimits_{\mathbb{K}\left[
t\right]  },D\right]  =0\ \ \ \ \ \ \ \ \ \ \text{and}%
\ \ \ \ \ \ \ \ \ \ \left[  \operatorname*{id}\nolimits_{\mathbb{K}\left[
t\right]  },T\right]  =0.
\]
Hence, we can obtain an example of the situation we are considering by setting
$\mathbb{L}=\operatorname*{End}\nolimits_{\mathbb{K}}\left(  \mathbb{K}\left[
t\right]  \right)  $, $a=D$, $x=T$ and $h=\operatorname*{id}%
\nolimits_{\mathbb{K}\left[  t\right]  }$.

Further examples can be obtained by varying this one. For example, if
$\mathbb{K}$ is a field, then $\mathbb{K}\left[  t\right]  $ can be replaced
by the field of rational functions $\mathbb{K}\left(  t\right)  $.
Alternatively, if $\mathbb{K}=\mathbb{R}$, then $\mathbb{K}\left[  t\right]  $
can be replaced by the algebra of $C^{\infty}$-functions $\mathbb{R}%
\rightarrow\mathbb{R}$.

Other examples appear in the theory of Weyl algebras and of $2$-step nilpotent
Lie algebras.
\end{example}

Now, we return to the generality of $\mathbb{K}$, $\mathbb{L}$, $a$, $h$, $x$
and $m$ satisfying $\left[  a,x\right]  =h$ and $\left[  h,a\right]  =0$ and
$\left[  h,x\right]  =0.$

For any polynomial $g\in\mathbb{K}\left[  t\right]  $, we define an $m\times
m$-matrix $V_{g}\in\mathbb{L}^{m\times m}$ by%
\begin{equation}
V_{g}=\left(
\begin{cases}
\dbinom{i}{j}g^{\left(  i-j\right)  }\left(  x\right)  \cdot h^{i-j}, &
\text{if }i\geq j;\\
0, & \text{if }i<j
\end{cases}
\right)  _{0\leq i<m,\ 0\leq j<m}. \label{eq.Vg=}%
\end{equation}
This matrix $V_{g}$ looks as follows:

\begin{itemize}
\item If $g\in\mathbb{K}\left[  t\right]  $ and $m\in\mathbb{N}$, then%
\[
V_{g}=\left(
\begin{smallmatrix}
g^{\left(  0\right)  }\left(  x\right)  & 0 & 0 & \cdots & 0\\
g^{\left(  1\right)  }\left(  x\right)  \cdot h & g^{\left(  0\right)
}\left(  x\right)  & 0 & \cdots & 0\\
g^{\left(  2\right)  }\left(  x\right)  \cdot h^{2} & 2g^{\left(  1\right)
}\left(  x\right)  \cdot h & g^{\left(  0\right)  }\left(  x\right)  & \cdots
& 0\\
\vdots & \vdots & \vdots & \ddots & \vdots\\
g^{\left(  m-1\right)  }\left(  x\right)  \cdot h^{m-1} & \left(  m-1\right)
g^{\left(  m-2\right)  }\left(  x\right)  \cdot h^{m-2} & \tbinom{m-1}%
{2}g^{\left(  m-3\right)  }\left(  x\right)  \cdot h^{m-3} & \cdots &
g^{\left(  0\right)  }\left(  x\right)
\end{smallmatrix}
\right)  .
\]

\item If $g\in\mathbb{K}\left[  t\right]  $ and $m=\infty$, then%
\[
V_{g}=\left(
\begin{array}
[c]{ccccc}%
g^{\left(  0\right)  }\left(  x\right)  & 0 & 0 & 0 & \cdots\\
g^{\left(  1\right)  }\left(  x\right)  \cdot h & g^{\left(  0\right)
}\left(  x\right)  & 0 & 0 & \cdots\\
g^{\left(  2\right)  }\left(  x\right)  \cdot h^{2} & 2g^{\left(  1\right)
}\left(  x\right)  \cdot h & g^{\left(  0\right)  }\left(  x\right)  & 0 &
\cdots\\
g^{\left(  3\right)  }\left(  x\right)  \cdot h^{3} & 3g^{\left(  2\right)
}\left(  x\right)  \cdot h^{2} & 3g^{\left(  1\right)  }\left(  x\right)
\cdot h & g^{\left(  0\right)  }\left(  x\right)  & \cdots\\
\vdots & \vdots & \vdots & \vdots & \ddots
\end{array}
\right)  .
\]

\end{itemize}

Now, Tom Copeland has found the following identity \cite{MO337766}:

\begin{theorem}
\label{thm.copeland}Let $n\in\mathbb{N}$ be such that $n<m$. Let
$g\in\mathbb{K}\left[  t\right]  $. Then,%
\[
\left(  g\left(  x\right)  \cdot a\right)  ^{n}=e_{0}^{T}\left(
V_{g}S\right)  ^{n}H_{1}.
\]

\end{theorem}

This identity will easily follow from Theorem \ref{thm.gen} (applied to
$b=g\left(  x\right)  $), once we can show the following:

\begin{proposition}
\label{prop.copeland-main}Let $g\in\mathbb{K}\left[  t\right]  $. Then,
$U_{g\left(  x\right)  }=V_{g}$.
\end{proposition}

We shall thus mostly focus on proving Proposition \ref{prop.copeland-main}.

\subsection{How derivatives appear in commutators}

The main idea of our proof will be the following proposition, which relates
derivatives in $\mathbb{K}\left[  t\right]  $ to commutators in $\mathbb{L}$:

\begin{proposition}
\label{prop.deriv1}\textbf{(a)} We have $ax^{i}=x^{i}a+ix^{i-1}h$ for each
positive integer $i$.

\textbf{(b)} We have $a\cdot g\left(  x\right)  =g\left(  x\right)  \cdot
a+g^{\prime}\left(  x\right)  \cdot h$ for each $g\in\mathbb{K}\left[
t\right]  $.
\end{proposition}

Here, of course, $g^{\prime}$ means the derivative $\dfrac{d}{dt}g$ of the
polynomial $g$.

\begin{proof}
[Proof of Proposition \ref{prop.deriv1}.]The definition of $\left[
a,x\right]  $ yields $\left[  a,x\right]  =ax-xa$. Hence, $ax-xa=\left[
a,x\right]  =h$. Thus, $ax=xa+h$.

From $\left[  h,x\right]  =0$, we obtain $0=\left[  h,x\right]  =hx-xh$ (by
the definition of $\left[  h,x\right]  $). In other words, $hx=xh$.

\textbf{(a)} We shall prove Proposition \ref{prop.deriv1} \textbf{(a)} by
induction on $i$:

\textit{Induction base:} Comparing $a\underbrace{x^{1}}_{=x}=ax=xa+h$ with
$\underbrace{x^{1}}_{=x}a+1\underbrace{x^{1-1}}_{=x^{0}=1}h=xa+h$, we find
$ax^{1}=x^{1}a+1x^{1-1}h$. In other words, Proposition \ref{prop.deriv1}
\textbf{(a)} holds for $i=1$. This completes the induction base.

\textit{Induction step:} Let $n$ be a positive integer. Assume that
Proposition \ref{prop.deriv1} \textbf{(a)} holds for $i=n$. We must prove that
Proposition \ref{prop.deriv1} \textbf{(a)} holds for $i=n+1$.

We have assumed that Proposition \ref{prop.deriv1} \textbf{(a)} holds for
$i=n$. In other words, we have $ax^{n}=x^{n}a+nx^{n-1}h$.

Now,
\begin{align*}
a\underbrace{x^{n+1}}_{=x^{n}x}  &  =\underbrace{ax^{n}}_{=x^{n}a+nx^{n-1}%
h}x=\left(  x^{n}a+nx^{n-1}h\right)  x=x^{n}\underbrace{ax}_{=xa+h}%
+nx^{n-1}\underbrace{hx}_{=xh}\\
&  =\underbrace{x^{n}\left(  xa+h\right)  }_{=x^{n}xa+x^{n}h}%
+n\underbrace{x^{n-1}x}_{=x^{n}}h=\underbrace{x^{n}x}_{=x^{n+1}}%
a+\underbrace{x^{n}h+nx^{n}h}_{=\left(  n+1\right)  x^{n}h}\\
&  =x^{n+1}a+\left(  n+1\right)  \underbrace{x^{n}}_{=x^{\left(  n+1\right)
-1}}h=x^{n+1}a+\left(  n+1\right)  x^{\left(  n+1\right)  -1}h.
\end{align*}
In other words, Proposition \ref{prop.deriv1} \textbf{(a)} holds for $i=n+1$.
This completes the induction step. Thus, Proposition \ref{prop.deriv1}
\textbf{(a)} is proven by induction.

\textbf{(b)} Let $g\in\mathbb{K}\left[  t\right]  $. Write the polynomial $g$
in the form $g=\sum_{i=0}^{k}g_{i}t^{i}$ for some $k\in\mathbb{N}$ and some
$g_{0},g_{1},\ldots,g_{k}\in\mathbb{K}$. Thus, the definition of the
derivative $g^{\prime}$ yields $g^{\prime}=\sum_{i=1}^{k}ig_{i}t^{i-1}$.
Substituting $x$ for $t$ in this equality, we find%
\begin{equation}
g^{\prime}\left(  x\right)  =\sum_{i=1}^{k}\underbrace{ig_{i}}_{=g_{i}%
i}x^{i-1}=\sum_{i=1}^{k}g_{i}ix^{i-1}. \label{pf.prop.deriv1.b.g'(x)}%
\end{equation}

Substituting $x$ for $t$ in the equality $g=\sum_{i=0}^{k}g_{i}t^{i}$, we
obtain $g\left(  x\right)  =\sum_{i=0}^{k}g_{i}x^{i}$. Hence,%
\begin{align*}
a\cdot\underbrace{g\left(  x\right)  }_{=\sum_{i=0}^{k}g_{i}x^{i}}  &
=a\cdot\sum_{i=0}^{k}g_{i}x^{i}=\sum_{i=0}^{k}g_{i}ax^{i}=g_{0}%
a\underbrace{x^{0}}_{=1}+\sum_{i=1}^{k}g_{i}\underbrace{ax^{i}}%
_{\substack{=x^{i}a+ix^{i-1}h\\\text{(by Proposition \ref{prop.deriv1}
\textbf{(a)})}}}\\
&  \ \ \ \ \ \ \ \ \ \ \left(  \text{here, we have split off the addend for
}i=0\text{ from the sum}\right) \\
&  =g_{0}a+\underbrace{\sum_{i=1}^{k}g_{i}\left(  x^{i}a+ix^{i-1}h\right)
}_{=\sum_{i=1}^{k}g_{i}x^{i}a+\sum_{i=1}^{k}g_{i}ix^{i-1}h}=g_{0}a+\sum
_{i=1}^{k}g_{i}x^{i}a+\sum_{i=1}^{k}g_{i}ix^{i-1}h.
\end{align*}
Comparing this with%
\begin{align*}
&  \underbrace{g\left(  x\right)  }_{=\sum_{i=0}^{k}g_{i}x^{i}}\cdot
a+\underbrace{g^{\prime}\left(  x\right)  }_{\substack{=\sum_{i=1}^{k}%
g_{i}ix^{i-1}\\\text{(by (\ref{pf.prop.deriv1.b.g'(x)}))}}}\cdot h\\
&  =\left(  \sum_{i=0}^{k}g_{i}x^{i}\right)  \cdot a+\left(  \sum_{i=1}%
^{k}g_{i}ix^{i-1}\right)  \cdot h=\underbrace{\sum_{i=0}^{k}g_{i}x^{i}%
a}_{\substack{=g_{0}x^{0}a+\sum_{i=1}^{k}g_{i}x^{i}a\\\text{(here, we have
split off the}\\\text{addend for }i=0\text{ from the sum)}}}+\sum_{i=1}%
^{k}g_{i}ix^{i-1}h\\
&  =g_{0}\underbrace{x^{0}}_{=1}a+\sum_{i=1}^{k}g_{i}x^{i}a+\sum_{i=1}%
^{k}g_{i}ix^{i-1}h=g_{0}a+\sum_{i=1}^{k}g_{i}x^{i}a+\sum_{i=1}^{k}%
g_{i}ix^{i-1}h,
\end{align*}
we obtain $a\cdot g\left(  x\right)  =g\left(  x\right)  \cdot a+g^{\prime
}\left(  x\right)  \cdot h$. This proves Proposition \ref{prop.deriv1}
\textbf{(b)}.
\end{proof}

Note that we have not used the condition $\left[  h,a\right]  =0$ in
Proposition \ref{prop.deriv1}; but we will use it now:

\begin{proposition}
\label{prop.deriv.triv-comm}Let $b\in\mathbb{L}$. Then, $\operatorname*{ad}%
\nolimits_{a}\left(  bh^{i}\right)  =\operatorname*{ad}\nolimits_{a}\left(
b\right)  \cdot h^{i}$ for each $i\in\mathbb{N}$.
\end{proposition}

\begin{proof}
[Proof of Proposition \ref{prop.deriv.triv-comm}.]From $\left[  h,a\right]
=0$, we obtain $0=\left[  h,a\right]  =ha-ah$ (by the definition of $\left[
h,a\right]  $). In other words, $ha=ah$. Hence,%
\begin{equation}
h^{i}a=ah^{i}\ \ \ \ \ \ \ \ \ \ \text{for each }i\in\mathbb{N}\text{.}
\label{pf.prop.deriv.triv-comm.hia}%
\end{equation}

\begin{vershort}
[\textit{Proof of (\ref{pf.prop.deriv.triv-comm.hia}):} This follows by
induction on $i$, using $ha=ah$ in the induction step.]
\end{vershort}

\begin{verlong}
[\textit{Proof of (\ref{pf.prop.deriv.triv-comm.hia}):} We shall prove
(\ref{pf.prop.deriv.triv-comm.hia}) by induction on $i$:

\textit{Induction base:} Comparing $\underbrace{h^{0}}_{=1}a=a$ with
$a\underbrace{h^{0}}_{=1}=a$, we obtain $h^{0}a=ah^{0}$. In other words,
(\ref{pf.prop.deriv.triv-comm.hia}) holds for $i=0$. This completes the
induction base.

\textit{Induction step:} Fix $j\in\mathbb{N}$. Assume that
(\ref{pf.prop.deriv.triv-comm.hia}) holds for $i=j$. We must prove that
(\ref{pf.prop.deriv.triv-comm.hia}) holds for $i=j+1$.

We have assumed that (\ref{pf.prop.deriv.triv-comm.hia}) holds for $i=j$. In
other words, we have $h^{j}a=ah^{j}$. Now,%
\[
\underbrace{h^{j+1}}_{=hh^{j}}a=h\underbrace{h^{j}a}_{=ah^{j}}=\underbrace{ha}%
_{=ah}h^{j}=a\underbrace{hh^{j}}_{=h^{j+1}}=ah^{j+1}.
\]
In other words, (\ref{pf.prop.deriv.triv-comm.hia}) holds for $i=j+1$. This
completes the induction step. Thus, (\ref{pf.prop.deriv.triv-comm.hia}) is
proven by induction.]
\end{verlong}

Now, let $i\in\mathbb{N}$. Then, the definition of $\operatorname*{ad}%
\nolimits_{a}$ yields%
\[
\operatorname*{ad}\nolimits_{a}\left(  b\right)  =\left[  a,b\right]
=ab-ba\ \ \ \ \ \ \ \ \ \ \left(  \text{by the definition of }\left[
a,b\right]  \right)  .
\]
But the definition of $\operatorname*{ad}\nolimits_{a}$ also yields%
\begin{align*}
\operatorname*{ad}\nolimits_{a}\left(  bh^{i}\right)   &  =\left[
a,bh^{i}\right]  =a\left(  bh^{i}\right)  -\left(  bh^{i}\right)
a\ \ \ \ \ \ \ \ \ \ \left(  \text{by the definition of }\left[
a,bh^{i}\right]  \right) \\
&  =abh^{i}-b\underbrace{h^{i}a}_{\substack{=ah^{i}\\\text{(by
(\ref{pf.prop.deriv.triv-comm.hia}))}}}=abh^{i}-bah^{i}=\underbrace{\left(
ab-ba\right)  }_{=\operatorname*{ad}\nolimits_{a}\left(  b\right)  }h^{i}\\
&  =\operatorname*{ad}\nolimits_{a}\left(  b\right)  \cdot h^{i}.
\end{align*}
This proves Proposition \ref{prop.deriv.triv-comm}.
\end{proof}

\begin{corollary}
\label{cor.deriv-ad}Let $g\in\mathbb{K}\left[  t\right]  $. Let $p\in
\mathbb{N}$. Then,%
\[
\operatorname*{ad}\nolimits_{a}^{p}\left(  g\left(  x\right)  \right)
=g^{\left(  p\right)  }\left(  x\right)  \cdot h^{p}.
\]

\end{corollary}

\begin{proof}
[Proof of Corollary \ref{cor.deriv-ad}.]We shall prove Corollary
\ref{cor.deriv-ad} by induction on $p$:

\textit{Induction base:} Comparing $\underbrace{\operatorname*{ad}%
\nolimits_{a}^{0}}_{=\operatorname*{id}}\left(  g\left(  x\right)  \right)
=\operatorname*{id}\left(  g\left(  x\right)  \right)  =g\left(  x\right)  $
with $\underbrace{g^{\left(  0\right)  }}_{=g}\left(  x\right)  \cdot
\underbrace{h^{0}}_{=1}=g\left(  x\right)  $, we obtain $\operatorname*{ad}%
\nolimits_{a}^{0}\left(  g\left(  x\right)  \right)  =g^{\left(  0\right)
}\left(  x\right)  \cdot h^{0}$. In other words, Corollary \ref{cor.deriv-ad}
holds for $p=0$. This completes the induction base.

\textit{Induction step:} Let $n\in\mathbb{N}$. Assume that Corollary
\ref{cor.deriv-ad} holds for $p=n$. We must prove that Corollary
\ref{cor.deriv-ad} holds for $p=n+1$.

We have assumed that Corollary \ref{cor.deriv-ad} holds for $p=n$. In other
words, we have%
\[
\operatorname*{ad}\nolimits_{a}^{n}\left(  g\left(  x\right)  \right)
=g^{\left(  n\right)  }\left(  x\right)  \cdot h^{n}.
\]

Now,%
\begin{align}
\underbrace{\operatorname*{ad}\nolimits_{a}^{n+1}}_{=\operatorname*{ad}%
\nolimits_{a}\circ\operatorname*{ad}\nolimits_{a}^{n}}\left(  g\left(
x\right)  \right)   &  =\left(  \operatorname*{ad}\nolimits_{a}\circ
\operatorname*{ad}\nolimits_{a}^{n}\right)  \left(  g\left(  x\right)
\right)  =\operatorname*{ad}\nolimits_{a}\left(
\underbrace{\operatorname*{ad}\nolimits_{a}^{n}\left(  g\left(  x\right)
\right)  }_{=g^{\left(  n\right)  }\left(  x\right)  \cdot h^{n}}\right)
\nonumber\\
&  =\operatorname*{ad}\nolimits_{a}\left(  g^{\left(  n\right)  }\left(
x\right)  \cdot h^{n}\right)  =\operatorname*{ad}\nolimits_{a}\left(
g^{\left(  n\right)  }\left(  x\right)  \right)  \cdot h^{n}
\label{pf.cor.deriv-ad.2}%
\end{align}
(by Proposition \ref{prop.deriv.triv-comm}, applied to $b=g^{\left(  n\right)
}\left(  x\right)  $ and $i=n$).

But Proposition \ref{prop.deriv1} \textbf{(b)} (applied to $g^{\left(
n\right)  }$ instead of $g$) yields
\[
a\cdot g^{\left(  n\right)  }\left(  x\right)  =g^{\left(  n\right)  }\left(
x\right)  \cdot a+\left(  g^{\left(  n\right)  }\right)  ^{\prime}\left(
x\right)  \cdot h.
\]
In view of $\left(  g^{\left(  n\right)  }\right)  ^{\prime}=g^{\left(
n+1\right)  }$ (this follows from the definitions of $g^{\left(  n\right)  }$
and $g^{\left(  n+1\right)  }$), we can rewrite this as%
\begin{equation}
a\cdot g^{\left(  n\right)  }\left(  x\right)  =g^{\left(  n\right)  }\left(
x\right)  \cdot a+g^{\left(  n+1\right)  }\left(  x\right)  \cdot h.
\label{pf.cor.deriv-ad.4}%
\end{equation}
Now, the definition of $\operatorname*{ad}\nolimits_{a}$ yields%
\begin{align*}
\operatorname*{ad}\nolimits_{a}\left(  g^{\left(  n\right)  }\left(  x\right)
\right)   &  =\left[  a,g^{\left(  n\right)  }\left(  x\right)  \right]
=a\cdot g^{\left(  n\right)  }\left(  x\right)  -g^{\left(  n\right)  }\left(
x\right)  \cdot a\\
&  \ \ \ \ \ \ \ \ \ \ \left(  \text{by the definition of }\left[
a,g^{\left(  n\right)  }\left(  x\right)  \right]  \right) \\
&  =g^{\left(  n+1\right)  }\left(  x\right)  \cdot
h\ \ \ \ \ \ \ \ \ \ \left(  \text{by (\ref{pf.cor.deriv-ad.4})}\right)  .
\end{align*}
Hence, (\ref{pf.cor.deriv-ad.2}) becomes%
\[
\operatorname*{ad}\nolimits_{a}^{n+1}\left(  g\left(  x\right)  \right)
=\underbrace{\operatorname*{ad}\nolimits_{a}\left(  g^{\left(  n\right)
}\left(  x\right)  \right)  }_{=g^{\left(  n+1\right)  }\left(  x\right)
\cdot h}\cdot h^{n}=g^{\left(  n+1\right)  }\left(  x\right)  \cdot
\underbrace{h\cdot h^{n}}_{=h^{n+1}}=g^{\left(  n+1\right)  }\left(  x\right)
\cdot h^{n+1}.
\]
In other words, Corollary \ref{cor.deriv-ad} holds for $p=n+1$. This completes
the induction step. Thus, Corollary \ref{cor.deriv-ad} is proven by induction.
\end{proof}

\subsection{Proofs of Proposition \ref{prop.copeland-main} and Theorem
\ref{thm.copeland}}

\begin{proof}
[Proof of Proposition \ref{prop.copeland-main}.]If $b\in\mathbb{L}$ is
arbitrary, then%
\begin{equation}
\left(  U_{b}\right)  _{i,j}=%
\begin{cases}
\dbinom{i}{j}\operatorname*{ad}\nolimits_{a}^{i-j}\left(  b\right)  , &
\text{if }i\geq j;\\
0, & \text{if }i<j
\end{cases}
\label{pf.prop.copeland-main.Ubij=}%
\end{equation}
for each $i\in\left\{  0,1,\ldots,m-1\right\}  $ and $j\in\left\{
0,1,\ldots,m-1\right\}  $ (by (\ref{eq.Ub=})).

On the other hand, (\ref{eq.Vg=}) shows that%
\begin{equation}
\left(  V_{g}\right)  _{i,j}=%
\begin{cases}
\dbinom{i}{j}g^{\left(  i-j\right)  }\left(  x\right)  \cdot h^{i-j}, &
\text{if }i\geq j;\\
0, & \text{if }i<j
\end{cases}
\label{pf.prop.copeland-main.Vgij=}%
\end{equation}
for each $i\in\left\{  0,1,\ldots,m-1\right\}  $ and $j\in\left\{
0,1,\ldots,m-1\right\}  $.

Now, let us fix $i\in\left\{  0,1,\ldots,m-1\right\}  $ and $j\in\left\{
0,1,\ldots,m-1\right\}  $. We shall prove that $\left(  U_{g\left(  x\right)
}\right)  _{i,j}=\left(  V_{g}\right)  _{i,j}$.

Indeed, we are in one of the following two cases:

\textit{Case 1:} We have $i\geq j$.

\textit{Case 2:} We have $i<j$.

Let us first consider Case 1. In this case, we have $i\geq j$. Hence,
$i-j\in\mathbb{N}$. Thus, Corollary \ref{cor.deriv-ad} (applied to $p=i-j$)
yields
\[
\operatorname*{ad}\nolimits_{a}^{i-j}\left(  g\left(  x\right)  \right)
=g^{\left(  i-j\right)  }\left(  x\right)  \cdot h^{i-j}.
\]
Now, (\ref{pf.prop.copeland-main.Ubij=}) (applied to $b=g\left(  x\right)  $)
yields%
\begin{align*}
\left(  U_{g\left(  x\right)  }\right)  _{i,j}  &  =%
\begin{cases}
\dbinom{i}{j}\operatorname*{ad}\nolimits_{a}^{i-j}\left(  g\left(  x\right)
\right)  , & \text{if }i\geq j;\\
0, & \text{if }i<j
\end{cases}
=\dbinom{i}{j}\underbrace{\operatorname*{ad}\nolimits_{a}^{i-j}\left(
g\left(  x\right)  \right)  }_{=g^{\left(  i-j\right)  }\left(  x\right)
\cdot h^{i-j}}\ \ \ \ \ \ \ \ \ \ \left(  \text{since }i\geq j\right) \\
&  =\dbinom{i}{j}g^{\left(  i-j\right)  }\left(  x\right)  \cdot h^{i-j}.
\end{align*}
Comparing this with%
\begin{align*}
\left(  V_{g}\right)  _{i,j}  &  =%
\begin{cases}
\dbinom{i}{j}g^{\left(  i-j\right)  }\left(  x\right)  \cdot h^{i-j}, &
\text{if }i\geq j;\\
0, & \text{if }i<j
\end{cases}
\ \ \ \ \ \ \ \ \ \ \left(  \text{by (\ref{pf.prop.copeland-main.Vgij=}%
)}\right) \\
&  =\dbinom{i}{j}g^{\left(  i-j\right)  }\left(  x\right)  \cdot
h^{i-j}\ \ \ \ \ \ \ \ \ \ \left(  \text{since }i\geq j\right)  ,
\end{align*}
we obtain $\left(  U_{g\left(  x\right)  }\right)  _{i,j}=\left(
V_{g}\right)  _{i,j}$. Thus, $\left(  U_{g\left(  x\right)  }\right)
_{i,j}=\left(  V_{g}\right)  _{i,j}$ is proven in Case 1.

Let us next consider Case 2. In this case, we have $i<j$. Hence,
(\ref{pf.prop.copeland-main.Vgij=}) becomes%
\[
\left(  V_{g}\right)  _{i,j}=%
\begin{cases}
\dbinom{i}{j}g^{\left(  i-j\right)  }\left(  x\right)  \cdot h^{i-j}, &
\text{if }i\geq j;\\
0, & \text{if }i<j
\end{cases}
=0\ \ \ \ \ \ \ \ \ \ \left(  \text{since }i<j\right)  .
\]
But (\ref{pf.prop.copeland-main.Ubij=}) (applied to $b=g\left(  x\right)  $)
yields%
\[
\left(  U_{g\left(  x\right)  }\right)  _{i,j}=%
\begin{cases}
\dbinom{i}{j}\operatorname*{ad}\nolimits_{a}^{i-j}\left(  g\left(  x\right)
\right)  , & \text{if }i\geq j;\\
0, & \text{if }i<j
\end{cases}
=0\ \ \ \ \ \ \ \ \ \ \left(  \text{since }i<j\right)  .
\]
Comparing these two equalities, we find $\left(  U_{g\left(  x\right)
}\right)  _{i,j}=\left(  V_{g}\right)  _{i,j}$. Thus, $\left(  U_{g\left(
x\right)  }\right)  _{i,j}=\left(  V_{g}\right)  _{i,j}$ is proven in Case 2.

We have now proven the equality $\left(  U_{g\left(  x\right)  }\right)
_{i,j}=\left(  V_{g}\right)  _{i,j}$ in both Cases 1 and 2. Hence, this
equality always holds.

Now, forget that we fixed $i$ and $j$. We thus have shown that $\left(
U_{g\left(  x\right)  }\right)  _{i,j}=\left(  V_{g}\right)  _{i,j}$ for all
$i\in\left\{  0,1,\ldots,m-1\right\}  $ and $j\in\left\{  0,1,\ldots
,m-1\right\}  $. In other words, each entry of the $m\times m$-matrix
$U_{g\left(  x\right)  }$ equals the corresponding entry of the $m\times
m$-matrix $V_{g}$. Hence, $U_{g\left(  x\right)  }=V_{g}$. This proves
Proposition \ref{prop.copeland-main}.
\end{proof}

\begin{proof}
[Proof of Theorem \ref{thm.copeland}.]Theorem \ref{thm.gen} (applied to
$b=g\left(  x\right)  $) yields%
\[
\left(  g\left(  x\right)  \cdot a\right)  ^{n}=e_{0}^{T}\left(
\underbrace{U_{g\left(  x\right)  }}_{\substack{=V_{g}\\\text{(by Proposition
\ref{prop.copeland-main})}}}S\right)  ^{n}H_{1}=e_{0}^{T}\left(
V_{g}S\right)  ^{n}H_{1}.
\]
This proves Theorem \ref{thm.copeland}.
\end{proof}


\begin{thebibliography}{99999999}                                                                                         %


\bibitem[MO337766]{MO337766}Tom Copeland, \textit{MathOverflow question
\#337766 (\textquotedblleft Expansions of iterated, or nested, derivatives, or
vectors -- conjectured matrix computation\textquotedblright)}. \newline\url{https://mathoverflow.net/questions/337766/}
\end{thebibliography}
\end{document}